\documentclass[10pt]{article}
\usepackage{graphicx}
\usepackage{epsfig}
\usepackage{amssymb}
\usepackage{amsmath}
\usepackage{amsthm}
\newtheorem{lemma}{Lemma}[section]

{\end{description}}
{\end{description}}

{\hfill $\bullet$ \\}

\newcommand{\be}{\begin{equation}}
\newcommand{\ee}{\end{equation}}

\begin{document}
    \title{A Novel Three-Level Time-Split MacCormack Method for Solving Two-Dimensional Viscous Coupled Burgers Equations}
   \author{\Large{Eric Ngondiep}
       \thanks{Tel.: +966501071861. E-mail addresses:\ ericngondiep@gmail.com or engondiep@imamu.edu.sa
        (Eric Ngondiep).\ }}
   \date{\small{Department of Mathematics and Statistics, College of Science, Imam Muhammad
   Ibn Saud Islamic University (IMSIU), 90950 Riyadh, Saudi Arabia}\\
       \text{\,}\\
       \small{Hydrological Research Centre, Institute for Geological and Mining Research, 4110 Yaounde-Cameroon}}

    \maketitle
   \textbf{Abstract.}
    In this paper, we analyze the three-level explicit time-split MacCormack procedure in the numerical solutions of two-dimensional viscous coupled Burgers' equations subject to initial and boundary conditions. The differential operators split the two-dimensional problem into two pieces so that the two-step explicit MacCormack scheme can be easily applied to each subproblem. This reduces the computational cost of the algorithm. For low Reynolds numbers, the proposed method is second order accurate in time and fourth convergent in space, while it is second order convergent in both time and space for high Reynolds numbers problems. This shows the efficiency and effectiveness of the considered method compared to a large set of numerical schemes widely studied in the literature for solving the two-dimensional time dependent nonlinear coupled Burgers' equations. Numerical examples which confirm the theoretical results are presented and discussed.
   \text{\,} \\
    \text{\,}
   \\
   {\bf AMS Subject Classification (MSC). 65M10, 65M05}.

      \section{Introduction and motivation}\label{sec1}
      A broad range of nonlinear evolutionary partial differential equations (PDEs) arise in several fields of science, namely in physics, engineering, chemistry, biology, finance and are very important in the mathematical formulation of continuum models. Moreover, systems of nonlinear PDEs have attracted much attention in the study of nonlinear time dependent equations describing wave propagation. For instance, the two-dimensional unsteady nonlinear coupled Burgers' equations are such type of PDEs. Burgers' equations occur in a wide area of applied mathematics such as, heat conduction, modeling of dynamics, acoustic wave, turbulent fluids and in continuous stochastic processes \cite{3zsd,1zsd,2zsd,bj}. Numerical analysis of Burgers' equations have attracted attention during the last decades and it represents an active filed of research to develop fast and efficient numerical schemes in the approximate solutions of such equations. In this paper, we should analyze the following two-dimensional evolutionary viscous coupled Burgers' equations
      \begin{equation}\label{1}
        u_{t}+uu_{x}+vu_{y}=\frac{1}{R}(u_{xx}+u_{yy}),\text{\,\,\,\,}v_{t}+vv_{x}+uv_{y}=\frac{1}{R}(v_{xx}+v_{yy}),\text{\,\,\,\,\,}
        (x,y)\in\Omega,\text{\,\,\,\,\,}t\in(0,T],
      \end{equation}
      with initial condition
      \begin{equation}\label{2}
        u(x,y,0)=u_{0}(x,y),\text{\,\,\,\,\,}v(x,y,0)=v_{0}(x,y),\text{\,\,\,\,\,}(x,y)\in\overline{\Omega},
      \end{equation}
      and boundary condition
      \begin{equation}\label{3}
        u(x,y,t)=\varphi_{1}(x,y,t),\text{\,\,\,\,\,}v(x,y,t)=\varphi_{2}(x,y,t),\text{\,\,\,\,\,}(x,y)\in\partial\Omega, \text{\,\,\,\,\,}t\in(0,T],
      \end{equation}
      where $u$ and $v$ are the unknown velocity, $R$ denotes the Reynolds numbers and $T>0,$ is the final time. Furthermore, $u_{t}$, $v_{t}$, $u_{x}$, $v_{x}$, $u_{y}$ and $v_{y}$ denote $\frac{\partial u}{\partial t}$, $\frac{\partial v}{\partial t}$, $\frac{\partial u}{\partial x}$,
      $\frac{\partial v}{\partial x}$, $\frac{\partial u}{\partial y}$ and $\frac{\partial v}{\partial y}$, respectively. $u_{xx}$, $u_{yy}$, $v_{xx}$ and $v_{yy}$ designate $\frac{\partial^{2}u}{\partial x^{2}}$, $\frac{\partial^{2}u}{\partial y^{2}}$, $\frac{\partial^{2}v}{\partial x^{2}}$ and $\frac{\partial^{2}v}{\partial y^{2}}$, respectively. $\Omega=(0,1)\times(0,1)$ is the fluid region, $\partial\Omega$ represents the boundary of $\Omega$. The initial conditions $u_{0}$ and $v_{0}$, and boundary conditions $\varphi_{j}$ $(j=1,2)$ are assumed to be smooth enough and satisfy the conditions $\varphi_{1}(x,y,0)=u_{0}(x,y),$ and $\varphi_{2}(x,y,0)=v_{0}(x,y),$ for any  $(x,y)\in\partial\Omega,$ so that the two-dimensional time-dependent nonlinear coupled equations $(\ref{1})$-$(\ref{3}),$ admits a smooth solution.\\

       In the literature, the two-dimensional nonstationary nonlinear coupled Burgers' equations $(\ref{1})$-$(\ref{3})$ have been solved using a large class of numerical schemes such as: spectral collocation methods, Adomain-pade approach, Fourier pseudospectral technique, Hopf-Code transformation, basis radial functions, fully implicit methods, multilevel ADI scheme, explicit-implicit procedure and domain decomposition method. For a survey of these methods, the readers should consult \cite{cf,kth,dhs,ri,en,md,go,ba,bj,el,cf1,lw,ds,dg,ked,mj,zsd} and references cited therein. For certain methods mentioned above, computations sometimes involving the nonlinear system of parabolic equations $(\ref{1})$-$(\ref{3})$ become unstable (that is, "blow up") because of the numerical oscillations. These oscillations are the results of inadequate mesh grid in regions of large gradients such as the shock waves and are accentuated when central differences are used for the first order spatial derivatives. However, because of the more restrictive stability conditions of the explicit schemes and the desire to maintain the tridiagonal matrices in the implicit methods, it is often important to modify some previous algorithms (those that provide good resolution in regions of large gradients) for multidimensional problems. For example, for nonlinear fluid flows, the two-step MacCormack procedure provides good resolution at discontinuities (see \cite{apt}, p. $187$, last paragraph).\\

        The explicit MacCormack approach is a suitable numerical scheme for solving nonlinear fluid flow equations. This technique is a predictor-corrector formulation in which the first order time derivative is approximated at the predictor and corrector stages using forward difference representations with alternate one-side differencing for the first order space derivatives. This is particularly useful for problems involving moving discontinuities. The best resolution of discontinuities occurs when the difference in the predictor arises in the direction of propagation of the discontinuity \cite{apt,30tls,19tls,22tls}. Furthermore, some previous works have indicated that this procedure is more suitable for systems of equations with nonlinear convective jacobian matrices using second order one-side explicit methods, such as Lax-Wendroff technique \cite{13tls,25tls}. For multidimensional problems, MacCormack has modified his original scheme by incorporating time splitting into the algorithm. The new method splits the explicit MacCormack into a sequence of one-dimensional operations, thereby, achieving a good stability condition. Hence, it advances the solution in each direction with the maximum allowable time step. For high Reynolds numbers flows where the viscous regions become very thin, the algorithm can be applied only in the coarse grid region. To overcome this challenge, MacCormack constructed a hybrid version of his technique, so called MacCormack rapid solver method. The rapid solver method is a combination of explicit MacCormack and an implicit scheme \cite{18tls}.\\

         Most recently, the author applied the rapid solver method of MacCormack and three-level explicit time-split MacCormack schemes in the approximate solutions of linear/nonlinear partial differential equations. Specifically, the hybrid version of MacCormack has been used to solve the mixed Stokes-Darcy and two-dimensional time-dependent incompressible Navier-Stokes equations while the three-level time-split MacCormack was applied to two-dimensional time-dependent reaction-diffusion, heat conduction, convection-diffusion equations and linear/nonlinear convection-diffusion-reaction equations with constant coefficients (diffusive term equals $1$ and convective velocity in the range: $-1$, $0.8$ and $1$). The analysis has suggested that the three-level explicit time-split MacCormack is fast, second order convergent in time and fourth order accurate in space \cite{en1ts,27tls,28tls,en2ts,tsmclcdr,en,26tls,29tls,tsmccd,tsmccdr}. We recall that the three-level time-split applies to a time dependent problem of the form: $u_{t}=A_{1}(u)+A_{2}(u),$ where $A_{j}$ ($j=1,2$) are differential operators, so that each subproblem $u_{t}=A_{j}(u)$, $j=1,2,$ is solved independently using the original MacCormack approach.\\

       In this paper, we analyze the three-level time-split MacCormack procedure for the two-dimensional evolutionary nonlinear coupled Burgers' equations $(\ref{1})$ subjects to the initial and boundary conditions $(\ref{2})$ and $(\ref{3})$, respectively. Specifically, this work is motivated by: (a) for low Reynolds numbers, the method is fast, efficient (explicit, second order accurate in time and fourth order convergent in space) and easy to implement under the time step requirement $\frac{2k}{Rh^{2}}\leq 1$, while for high Reynolds numbers and under the time step constraint $\frac{k^{\frac{3}{4}}}{h}\leq 1,$ the algorithm is efficient (explicit and second order accurate in both time and space); (b) the form of the time step restriction: $\max\left\{\frac{2k}{Rh^{2}},\frac{k^{\frac{3}{4}}}{h}\right\}\leq 1$. Indeed, equations $(\ref{1})$-$(\ref{3})$ can model nonlinear hyperbolic equations (for large Reynolds numbers $R$) and parabolic problem (case of low Reynolds numbers $R$). For instance, the time step limitation provided by the Fourier analysis for stability of explicit schemes when solving two-dimensional linear parabolic equations is given by $\frac{4ak}{h^{2}}\leq1,$ which is well known in the literature as the CFL condition. Regarding the two-dimensional unsteady nonlinear system of Burgers' equations, the classical Von Neumann stability analysis is not in the standard sense, directly applicable; (c) for low Reynolds number the corresponding time step requirement $\frac{2k}{Rh^{2}} \leq 1$, indicates that the solution is advanced in each direction with the maximum allowable time step. In order to describe this scheme, we consider the one-dimensional difference operators $L_{x}(\Delta t_{x})$ and $L_{y}(\Delta t_{y})$. The $L_{x}(\Delta t_{x})$ nonlinear operator applied to the vector
       $\begin{pmatrix}
                    u_{ij}^{n} \\
                    v_{ij}^{n} \\
                  \end{pmatrix}$,

      \begin{equation}\label{8}
        \begin{pmatrix}
                    u_{ij}^{*} \\
                    v_{ij}^{*} \\
                  \end{pmatrix}=L_{x}(\Delta t_{x})\begin{pmatrix}
                                                     u_{ij}^{n} \\
                                                     v_{ij}^{n} \\
                                                   \end{pmatrix},
       \end{equation}
      is by definition equivalent to the original MacCormack numerical scheme. In a like manner, the nonlinear operator $L_{y}(\Delta t_{y})$ is defined by
      \begin{equation}\label{9}
        \begin{pmatrix}
                    u_{ij}^{*} \\
                    v_{ij}^{*} \\
                  \end{pmatrix}=L_{y}(\Delta t_{x})\begin{pmatrix}
                                                     u_{ij}^{n} \\
                                                     v_{ij}^{n} \\
                                                   \end{pmatrix}.
       \end{equation}
       The expressions make use of a dummy time index, which is denoted by the asterisk. Putting $\Delta t_{y}=\Delta t$ and
        $\Delta t_{x}=\frac{\Delta t}{2p},$ where $p$ is a positive integer, a high order convergent scheme can be constructed by applying the
        $L_{x}$ and $Ly$ operators to
        $\begin{pmatrix}
                    u_{ij}^{n} \\
                    v_{ij}^{n} \\
                  \end{pmatrix},$ in the following manner
       \begin{equation*}
        \begin{pmatrix}
                    u_{ij}^{n+1} \\
                    v_{ij}^{n+1} \\
                  \end{pmatrix}=\left[L_{x}\left(\frac{\Delta t}{2p}\right)L_{y}\left(\frac{\Delta t}{p}\right)L_{x}\left(\frac{\Delta t}{2p}\right)\right]^{p}\begin{pmatrix}
                    u_{ij}^{n} \\
                    v_{ij}^{n} \\
                  \end{pmatrix}.
      \end{equation*}
      This formula is particularly important when solving nonlinear problems with high Reynolds numbers. In fact, for such problems the fine-grid region becomes very thin, requiring $\Delta t_{x}$ to be very small. This must cause $\Delta t_{x}$ in the $L_{x}$ operator to be small and the integer $p$ to be very large. As consequence, a substantial amount of computation time is required in the fine-grid region.\\

        In the remainder of this work, the $1$D operations $L_{x}(\Delta t_{x})$ and $L_{y}(\Delta t_{y})$ are defined by equations $(\ref{8})$ and $(\ref{9})$, respectively. Moreover, we set $p=1$ and a second-order accurate scheme (in time) can be constructed by applying the $L_{x}$ and $L_{y}$ operators to
        $\begin{pmatrix}
                    u_{ij}^{n} \\
                    v_{ij}^{n} \\
                  \end{pmatrix}$ in the following way
       \begin{equation}\label{10}
        \begin{pmatrix}
                    u_{ij}^{n+1} \\
                    v_{ij}^{n+1} \\
                  \end{pmatrix}=L_{x}\left(\frac{\Delta t}{2}\right)L_{y}(\Delta t)L_{x}\left(\frac{\Delta t}{2}\right)
                  \begin{pmatrix}
                    u_{ij}^{n} \\
                    v_{ij}^{n} \\
                  \end{pmatrix}.
       \end{equation}
       Setting $\Delta t_{y}=k,$ $\Delta t_{x}=\frac{k}{2}$ and $\Delta x=\Delta y:=h,$ equations $(\ref{8})$, $(\ref{9})$ and $(\ref{10})$ yield
       \begin{equation}\label{11}
       \begin{pmatrix}
                    u_{ij}^{*} \\
                    v_{ij}^{*} \\
                  \end{pmatrix}=L_{x}(k/2)\begin{pmatrix}
                                                     u_{ij}^{n} \\
                                                     v_{ij}^{n} \\
                                                   \end{pmatrix},\text{\,\,}\begin{pmatrix}
                    u_{ij}^{**} \\
                    v_{ij}^{**} \\
                  \end{pmatrix}=L_{y}(k)\begin{pmatrix}
                                                     u_{ij}^{*} \\
                                                     v_{ij}^{*} \\
                                                   \end{pmatrix}=L_{y}(k)L_{x}(k/2)\begin{pmatrix}
                                                     u_{ij}^{n} \\
                                                     v_{ij}^{n} \\
                                                   \end{pmatrix},\text{\,\,}\begin{pmatrix}
                    u_{ij}^{n+1} \\
                    v_{ij}^{n+1} \\
                  \end{pmatrix}=L_{x}(k/2)\begin{pmatrix}
                                                     u_{ij}^{**} \\
                                                     v_{ij}^{**} \\
                                                   \end{pmatrix}.
       \end{equation}
       To construct the algorithm, we should find simple expressions of systems of equations
       $\begin{pmatrix}
                    u_{ij}^{*} \\
                    v_{ij}^{*} \\
                  \end{pmatrix}=L_{x}(k/2)\begin{pmatrix}
                                                     u_{ij}^{n} \\
                                                     v_{ij}^{n} \\
                                                   \end{pmatrix}$ and $\begin{pmatrix}
                    u_{ij}^{**} \\
                    v_{ij}^{**} \\
                  \end{pmatrix}=L_{y}(k)\begin{pmatrix}
                                                     u_{ij}^{*} \\
                                                     v_{ij}^{*} \\
                                                   \end{pmatrix}.$
       This will be an important tool to get an explicit formula of system of equations
       $\begin{pmatrix}
                    u_{ij}^{n+1} \\
                    v_{ij}^{n+1} \\
                  \end{pmatrix}=L_{x}(k/2)\begin{pmatrix}
                                                     u_{ij}^{**} \\
                                                     v_{ij}^{**} \\
                                                   \end{pmatrix}$,
       which denotes a "one-step time-split MacCormack algorithm". For the sake of simplicity, we use notations: $w)_{ij}^{n}=w_{ij}^{n}$ and $[w+z]_{ij}^{n}=w_{ij}^{n}+z^{n}_{ij}$.\\

        In this paper, we are interested in the numerical solutions of the initial-boundary value problem $(\ref{1})$-$(\ref{3})$. Moreover, the work is focused of the following items:
       \begin{description}
         \item[1.] a full description of the three-level time-split MacCormack scheme applied to the two-dimensional time dependent nonlinear coupled Burgers' equations $(\ref{1})$-$(\ref{3});$
         \item[2.] a wide set of numerical examples which confirm the theoretical analysis are presented and critically discussed.
       \end{description}
         The two items represent our original contributions since to our knowledge there is no available result in the literature provided by a three-level explicit time-split MacCormack procedure in the numerical solution of nonstationary viscous coupled Burgers' equations $(\ref{1})$-$(\ref{3})$.\\

        The paper is organized as follows: In section $\ref{sec2},$ we describe the three-level time-split MacCormack method for solving the initial-boundary value problem $(\ref{1})$-$(\ref{3}).$ Section $\ref{sec3}$ deals with a large set of numerical experiments which confirm the theoretical analysis (stability and predicted convergence rate) of the proposed algorithm. We present in section $\ref{sec4}$ the general conclusions and future directions of works.

      \section{Detailed description of the three-level time-split scheme}\label{sec2}
        We develop a three-level explicit time-split MacCormack method for solving the two-dimensional unsteady nonlinear coupled Burgers equations $(\ref{1})$-$(\ref{3}).$\\

           Consider $N$ and $M$ be two positive integers. Set $k:=\Delta t=\frac{T}{N};$ $h:=\Delta x=\Delta y=\frac{1}{M},$ be the time step and grid spacing, respectively. Let $t^{n}=kn,$ $t^{*}=(n+r)k,$ $t^{**}=(n+s)k,$ where $0<r<s<1,$ so $t^{*}\in(t^{n},t^{n+1}),$ $t^{**}\in(t^{*},t^{n+1});$ $n=0,1,2,...,N-1;$ $x_{i}=ih;$ $y_{j}=jh;$ for $0\leq i,j\leq M$. Furthermore, we introduce the discrete regions $\Omega_{k}=\{t^{n},0\leq n\leq N\};$ $\overline{\Omega}_{h}=\{(x_{i},y_{j}),0\leq i,j\leq M\};$ $\Omega_{h}=\overline{\Omega}_{h}\cap\Omega$ and $\partial\Omega_{h}=\overline{\Omega}_{h}\cap\partial\Omega.$\\

          Suppose $\mathcal{U}_{h}=\{\phi_{ij}^{n}=\phi(x^{i},y^{j},t^{n}),n=0,1,...,N;\text{\,}0\leq i,j\leq M\},$ be the grid space of functions defined on $\Omega_{h}\times\Omega_{k}.$ Setting
       \begin{equation*}
        \delta_{t}\phi_{ij}^{*}=\frac{\phi_{ij}^{*}-\phi_{ij}^{n}}{k/2},\text{\,} \delta_{t} \phi_{ij}^{**}=\frac{\phi_{ij}^{**}-\phi_{ij}^{*}}{k},
        \text{\,}\delta_{t} \phi_{ij}^{n+1}=\frac{\phi_{ij}^{n+1}-\phi_{ij}^{**}}{k/2},\text{\,}\delta_{x}\phi_{i+\frac{1}{2},j}^{n}=\frac{\phi_{i+1,j}^{n}
        -\phi_{ij}^{n}}{h},\text{\,}\delta^{x}\phi_{ij}^{n}=\frac{\phi_{i+1,j}^{n}-\phi_{i-1,j}^{n}}{2h},
       \end{equation*}
       \begin{equation}\label{4}
        \delta_{y}\phi_{i,j+\frac{1}{2}}^{n}=\frac{\phi_{i,j+1}^{n}-\phi_{ij}^{n}}{h},\text{\,}\delta^{y}\phi_{ij}^{n}=\frac{\phi_{i,j+1}^{n}-
        \phi_{i,j-1}^{n}}{2h},
        \text{\,}\delta_{x}^{2}\phi_{ij}^{n}=\frac{\delta_{x}\phi_{i+\frac{1}{2},j}^{n}-\delta_{x}\phi_{i-\frac{1}{2},j}^{n}}{h},\text{\,}\delta_{y}^{2}
        \phi_{ij}^{n}=\frac{\delta_{y}\phi_{i,j+\frac{1}{2}}^{n}-\delta_{y}\phi_{i,j-\frac{1}{2}}^{n}}{h}.
       \end{equation}
         From the definition of the operators $\delta^{x}$ and $\delta^{y},$ it is easy to see that $\delta^{x}\phi_{ij}^{n}=\frac{1}{2}\left(\delta_{x}\phi_{i+\frac{1}{2},j}^{n}+\delta_{x}\phi_{i-\frac{1}{2},j}^{n}\right)$
       and $\delta^{y}\phi_{ij}^{n}=\frac{1}{2}\left(\delta_{y}\phi_{i,j+\frac{1}{2}}^{n}+\delta_{y}\phi_{i,j-\frac{1}{2}}^{n}\right)$.
       The discrete norms are defined by
       \begin{equation*}
        \|\phi^{n}\|_{L^{2}}=h\left(\underset{i,j=1}{\overset{M-1}\sum}|\phi_{ij}^{n}|^{2}\right)^{\frac{1}{2}},\text{\,\,}\|\delta_{x}\phi^{n}\|_{L^{2}}=
        h\left(\underset{j=1}{\overset{M-1}\sum}\underset{i=0}{\overset{M-1}\sum}|\delta_{x}\phi_{i+\frac{1}{2},j}^{n}|^{2}\right)^{\frac{1}{2}},
       \end{equation*}
       \begin{equation}\label{5}
        \|\delta_{y}\phi^{n}\|_{L^{2}}=h\left(\underset{j=0}{\overset{M-1}\sum}\underset{i=1}{\overset{M-1}\sum}|\delta_{y}\phi_{i,j
        +\frac{1}{2}}^{n}|^{2}\right)^{\frac{1}{2}},
        \text{\,}\|\delta_{\lambda}^{2}\phi^{n}\|_{L^{2}}=h\left(\underset{i,j=1}{\overset{M-1}\sum}|\delta_{\lambda}^{2}
        \phi_{ij}^{n}|^{2}\right)^{\frac{1}{2}},
       \end{equation}
       where $\lambda=x,y.$ Furthermore, the scalar products are defined as
       \begin{equation*}
        (\phi^{n},v^{n})=h^{2}\underset{i,j=1}{\overset{M-1}\sum}\phi^{n}_{ij}v_{ij}^{n}, \text{\,\,}<\delta_{x}\phi^{n},\delta_{x}v^{n}>_{x}=h^{2}\underset{j=1}
        {\overset{M-1}\sum}\underset{i=0}{\overset{M-1}\sum}\delta_{x}\phi_{i+\frac{1}{2},j}^{n}\delta_{x}v_{i+\frac{1}{2},j}^{n},
       \end{equation*}
       and
       \begin{equation}\label{6}
        <\delta_{y}\phi^{n},\delta_{y}v^{n}>_{y}=h^{2}\underset{j=0}{\overset{M-1}\sum}\underset{i=1}{\overset{M-1}\sum}\delta_{y}\phi_{i,j
        +\frac{1}{2}}^{n}\delta_{y}v_{i,j+\frac{1}{2}}^{n}.
       \end{equation}
       The sobolev space $L^{m}\left(0,T;L^{2}(\Omega)\right)$ ($m=1,2,\infty$) endowed with the norms $\||\cdot|\|_{L^{m}\left(0,T;L^{2}(\Omega)\right)}$, ($m=1,2,\infty$) are defined by
       \begin{equation*}
        \||w|\|_{L^{1}\left(0,T;L^{2}(\Omega)\right)}=k\underset{n=0}{\overset{N}\sum}\|w^{n}\|_{L^{2}},\text{\,\,\,}
        \||w|\|_{L^{2}\left(0,T;L^{2}(\Omega)\right)}=\left(k\underset{n=0}{\overset{N}\sum}\|w^{n}\|_{L^{2}}^{2}\right)^{\frac{1}{2}},\text{\,\,\,}
       \end{equation*}
       and
       \begin{equation}\label{7}
        \||w|\|_{L^{\infty}\left(0,T;L^{2}(\Omega)\right)}=\underset{0\leq n\leq N}{\max}\|w^{n}\|_{L^{2}}.
       \end{equation}
        We recall that an explicit time-split MacCormack \cite{en2ts,tsmccd} splits the original explicit MacCormack scheme into a sequence of one-dimensional operators, thereby achieving a good stability restriction. Moreover, the primary impetus in the development of the time-split algorithm is to reduce the amount of computational work to advance the solution one time step. For example, in the explicit methods, the splitting of the equations so that the various operators are advanced separately allows one, in principle, to advance each of these phases at its own stability limitation. Because the stability restriction for some of the steps can be substantially larger than for others in a typical high Reynolds numbers flow, a savings in computational effort can be realized (\cite{apt}, page $231$). In the following, we analyze the three-level explicit time-split MacCormack procedure in a numerical solution of the initial-boundary value problem $(\ref{1})$-$(\ref{3})$ under the time step requirement
        \begin{equation}\label{11**}
            \max\left\{\frac{2R^{-1}k}{h^{2}},\frac{k^{\frac{3}{4}}}{h}\right\}\leq 1,
        \end{equation}
        where $R$ denotes the Reynolds number.\\

        Expanding the Taylor series about $(x_{i},y_{j},t^{n})$ at the predictor and corrector phases with time step $k/2$ in a two-step explicit MacCormack technique to get
       \begin{equation}\label{12a}
        u_{ij}^{\overline{*}}=u^{n}_{ij}+\frac{k}{2}u_{t})_{ij}^{n}+O(k^{2}),\text{\,\,}v_{ij}^{\overline{*}}=v^{n}_{ij}+\frac{k}{2}v_{t})_{ij}^{n}+O(k^{2}),
        \text{\,\,}u_{ij}^{\overline{\overline{*}}}=u^{n}_{ij}+\frac{k}{2}u_{t})_{ij}^{\overline{*}}+O(k^{2}),\text{\,\,}v_{ij}^{\overline{\overline{*}}}=
        v^{n}_{ij}+\frac{k}{2}v_{t})_{ij}^{\overline{*}}+O(k^{2}).
       \end{equation}
        From the definition of the nonlinear operator $L_{x}(k/2),$ let consider the equations
       \begin{equation*}
        u_{t}+uu_{x}=\frac{1}{R}u_{xx}\text{\,\,and\,\,}v_{t}+uv_{x}=\frac{1}{R}v_{xx},
       \end{equation*}
       which are equivalent to
       \begin{equation}\label{11a}
        u_{t}=-uu_{x}+\frac{1}{R}u_{xx}\text{\,\,and\,\,}v_{t}=-uv_{x}+\frac{1}{R}v_{xx}.
       \end{equation}
         Substituting this and $(\ref{11a})$ into equation $(\ref{12a})$ provides
         \begin{equation}\label{12}
         u_{ij}^{\overline{*}}=u^{n}_{ij}+\frac{k}{2}[-uu_{x}+\frac{1}{R}u_{xx}]_{ij}^{n}+O(k^{2}),\text{\,\,}
         v_{ij}^{\overline{*}}=v^{n}_{ij}+\frac{k}{2}[-uv_{x}+\frac{1}{R}v_{xx}]_{ij}^{n}+O(k^{2}),
        \end{equation}
        and
        \begin{equation}\label{13}
        u_{ij}^{\overline{\overline{*}}}=u^{n}_{ij}+\frac{k}{2}[-uu_{x}+\frac{1}{R}u_{xx}]_{ij}^{\overline{*}}+O(k^{2}),\text{\,\,}
         v_{ij}^{\overline{\overline{*}}}=v^{n}_{ij}+\frac{k}{2}[-uv_{x}+\frac{1}{R}v_{xx}]_{ij}^{\overline{*}}+O(k^{2}).
       \end{equation}
       Applying the Taylor series expansion about $(x_{i},y_{j},t^{n})$ and $(x_{i},y_{j},t^{\overline{*}})$ with mesh size $h$ using both forward and backward difference representations to get
       \begin{equation*}
        u_{x,ij}^{n}=\delta_{x}u^{n}_{i+\frac{1}{2},j}+O(h),\text{\,}v_{x,ij}^{n}=\delta_{x}v^{n}_{i+\frac{1}{2},j}+O(h),
        \text{\,}u_{2x,ij}^{n}=\delta_{x}^{2}u^{n}_{ij}+O(h^{2}),\text{\,}v_{2x,ij}^{n}=\delta_{x}^{2}v^{n}_{ij}+O(h^{2}),
       \end{equation*}
       \begin{equation}\label{14}
        u_{x,ij}^{\overline{*}}=\delta_{x}u^{\overline{*}}_{i-\frac{1}{2},j}+O(h),\text{\,}v_{x,ij}^{\overline{*}}=\delta_{x}
        v^{\overline{*}}_{i-\frac{1}{2},j}+O(h),\text{\,}u_{2x,ij}^{\overline{*}}=\delta_{x}^{2}u^{\overline{*}}_{ij}+O(h^{2}),
        \text{\,}v_{2x,ij}^{\overline{*}}=\delta_{x}^{2}v^{\overline{*}}_{ij}+O(h^{2}).
       \end{equation}
        The linear operators $\delta_{x}$ and $\delta_{x}^{2}$ are given by equation $(\ref{4}).$ Plugging equations $(\ref{12})$, $(\ref{13})$ and $(\ref{14})$, direct computations result in
       \begin{equation}\label{15}
        u_{ij}^{\overline{*}}=u^{n}_{ij}+\frac{k}{2}\left[-u_{ij}^{n}\frac{u_{i+1,j}^{n}-u_{ij}^{n}}{h}+\frac{1}{R}
        \frac{u_{i+1,j}^{n}-2u_{ij}^{n}+u_{i-1,j}^{n}}{h^{2}}\right]+O(k^{2}+kh),
       \end{equation}
       \begin{equation}\label{16}
        v_{ij}^{\overline{*}}=v^{n}_{ij}+\frac{k}{2}\left[-u_{ij}^{n}\frac{v_{i+1,j}^{n}-v_{ij}^{n}}{h}+\frac{1}{R}
        \frac{v_{i+1,j}^{n}-2v_{ij}^{n}+v_{i-1,j}^{n}}{h^{2}}\right]+O(k^{2}+kh),
       \end{equation}
       and
       \begin{equation}\label{17}
        u_{ij}^{\overline{\overline{*}}}=u^{n}_{ij}+\frac{k}{2}\left[-u_{ij}^{\overline{*}}\frac{u_{ij}^{\overline{*}}-u_{i-1,j}^{\overline{*}}}{h}
        +\frac{1}{R}\frac{u_{i+1,j}^{\overline{*}}-2u_{ij}^{\overline{*}}+u_{i-1,j}^{\overline{*}}}{h^{2}}\right]+O(k^{2}+kh),
       \end{equation}
       \begin{equation}\label{18}
        v_{ij}^{\overline{\overline{*}}}=v^{n}_{ij}+\frac{k}{2}\left[-u_{ij}^{\overline{*}}\frac{v_{ij}^{\overline{*}}-v_{i-1,j}^{\overline{*}}}{h}
        +\frac{1}{R}\frac{v_{i+1,j}^{\overline{*}}-2v_{ij}^{\overline{*}}+v_{i-1,j}^{\overline{*}}}{h^{2}}\right]+O(k^{2}+kh).
       \end{equation}
        Taking the average of $u_{ij}^{\overline{*}}$ and $u_{ij}^{\overline{\overline{*}}}$ (respectively, $v_{ij}^{\overline{*}}$ and $v_{ij}^{\overline{\overline{*}}}$), it is easy to see that
        \begin{equation*}
        \frac{u_{ij}^{\overline{*}}+u_{ij}^{\overline{\overline{*}}}}{2}=u^{n}_{ij}+\frac{k}{4}\left[-\left(u_{ij}^{n}\frac{u_{i+1,j}^{n}-u_{ij}^{n}}{h}
        +u_{ij}^{\overline{*}}\frac{u_{ij}^{\overline{*}}-u_{i-1,j}^{\overline{*}}}{h}\right)
        +\frac{1}{R}\left(\frac{u_{i+1,j}^{n}-2u_{ij}^{n}+u_{i-1,j}^{n}}{h^{2}}+\right.\right.
       \end{equation*}
       \begin{equation}\label{19}
        \left.\left.\frac{u_{i+1,j}^{\overline{*}}-2u_{ij}^{\overline{*}}+u_{i-1,j}^{\overline{*}}}{h^{2}}\right)\right]+O(k^{2}+kh),
       \end{equation}
       and
       \begin{equation*}
        \frac{v_{ij}^{\overline{*}}+v_{ij}^{\overline{\overline{*}}}}{2}=v^{n}_{ij}+\frac{k}{4}\left[-\left(u_{ij}^{n}\frac{v_{i+1,j}^{n}-v_{ij}^{n}}{h}
        +u_{ij}^{\overline{*}}\frac{v_{ij}^{\overline{*}}-v_{i-1,j}^{\overline{*}}}{h}\right)
        +\frac{1}{R}\left(\frac{v_{i+1,j}^{n}-2v_{ij}^{n}+v_{i-1,j}^{n}}{h^{2}}+\right.\right.
       \end{equation*}
       \begin{equation}\label{20}
        \left.\left.\frac{v_{i+1,j}^{\overline{*}}-2v_{ij}^{\overline{*}}+v_{i-1,j}^{\overline{*}}}{h^{2}}\right)\right]+O(k^{2}+kh).
       \end{equation}
        In order to provide a detailed description of the three-level explicit time-split algorithm for solving the initial-boundary value problem  $(\ref{1})$-$(\ref{3})$, some intermediate results are needed. The following lemma considers such results.

        \begin{lemma}\label{l1}
        Let $u_{ij}^{n}=u(x_{i},y_{j},t^{n})$ and $v_{ij}^{n}=v(x_{i},y_{j},t^{n})$ be the solutions satisfying equations $(\ref{19})$ and $(\ref{20})$, respectively. Thus, it holds:
       \begin{equation}\label{21}
        u_{ij}^{\overline{*}}\delta_{x}u_{i-\frac{1}{2},j}^{\overline{*}}+u_{ij}^{n}\delta_{x}u_{i+\frac{1}{2},j}^{n}=2u_{ij}^{n}\delta^{x} u_{ij}^{n}
        +O(k+k^{2}h^{-1}),
       \end{equation}
       \begin{equation}\label{22}
        u_{ij}^{\overline{*}}\delta_{x}v_{i-\frac{1}{2},j}^{\overline{*}}+u_{ij}^{n}\delta_{x}v_{i+\frac{1}{2},j}^{n}=2u_{ij}^{n}\delta^{x} v_{ij}^{n}+O(k+k^{2}h^{-1}),
       \end{equation}
       \begin{equation}\label{23}
        \frac{1}{R}(\delta^{2}_{x}u_{ij}^{\overline{*}}+\delta^{2}_{x}u_{ij}^{n})=\frac{2}{R}\delta^{2}_{x}u_{ij}^{n}+O(k+k^{2}h^{-2}),
       \end{equation}
       and
       \begin{equation}\label{24}
        \frac{1}{R}(\delta^{2}_{x}v_{ij}^{\overline{*}}+\delta^{2}_{x}v_{ij}^{n})=\frac{2}{R}\delta^{2}_{x}v_{ij}^{n}+O(k+k^{2}h^{-2}),
       \end{equation}
       where the linear operators $\delta_{x},$ $\delta^{x}$ and $\delta^{2}_{x},$ are given by equation $(\ref{4}).$
       \end{lemma}

       \begin{proof}
       We must prove only equations $(\ref{21})$ and $(\ref{23}).$ The proof for equations $(\ref{22})$ and $(\ref{24})$ are similar.\\

       It comes from the definition of the linear operators $\delta_{x}$ and $\delta^{2}_{x}$ that
       \begin{equation}\label{25}
       \delta_{x}u_{i-\frac{1}{2},j}^{\overline{*}}=\frac{u_{ij}^{\overline{*}}-u_{i-1,j}^{\overline{*}}}{h},\text{\,\,}
       \delta_{x}^{2}u_{ij}^{\overline{*}}=\frac{u_{i+1,j}^{\overline{*}}-2u_{ij}^{\overline{*}}+u_{i-1,j}^{\overline{*}}}{h^{2}},
       \end{equation}
       and
       \begin{equation}\label{25*}
       \delta_{x}v_{i-\frac{1}{2},j}^{\overline{*}}=\frac{v_{ij}^{\overline{*}}-v_{i-1,j}^{\overline{*}}}{h},\text{\,\,}
       \delta_{x}^{2}v_{ij}^{\overline{*}}=\frac{v_{i+1,j}^{\overline{*}}-2v_{ij}^{\overline{*}}+v_{i-1,j}^{\overline{*}}}{h^{2}}.
       \end{equation}
       Combining equations $(\ref{25})$ and $(\ref{16}),$ straightforward calculations give
       \begin{equation*}
        u_{ij}^{\overline{*}}\delta_{x}u_{i-\frac{1}{2},j}^{\overline{*}}=\frac{1}{h}(u_{ij}^{\overline{*}}-u_{i-1,j}^{\overline{*}})
        \left\{u^{n}_{ij}+\frac{k}{2}\left[-u_{ij}^{n}\delta_{x}u_{i+\frac{1}{2},j}^{n}+\frac{1}{R}
        \delta_{x}^{2}u_{ij}^{n}\right]+O(k^{2}+kh)\right\}=
       \end{equation*}
       \begin{equation*}
        \frac{1}{h}\left\{u^{n}_{ij}+\frac{k}{2}\left[-u_{ij}^{n}\delta_{x}u_{i+\frac{1}{2},j}^{n}+\frac{1}{R}\delta_{x}^{2}u_{ij}^{n}\right]
        -u^{n}_{i-1,j}-\frac{k}{2}\left[-u_{i-1,j}^{n}\delta_{x}u_{i-\frac{1}{2},j}^{n}+\frac{1}{R}\delta_{x}^{2}u_{i-1,j}^{n}\right]
        +O(k^{2}+kh)\right\}\times
       \end{equation*}
       \begin{equation*}
       \left\{u^{n}_{ij}+\frac{k}{2}\left[-u_{ij}^{n}\delta_{x}u_{i+\frac{1}{2},j}^{n}+\frac{1}{R}\delta_{x}^{2}u_{ij}^{n}\right]
        +O(k^{2}+kh)\right\}=\left\{\delta_{x}u_{i-\frac{1}{2},j}^{n}+\frac{k}{2h}\left[-u_{ij}^{n}\delta_{x}u_{i+\frac{1}{2},j}^{n}
        +u_{i-1,j}^{n}\delta_{x}u_{i-\frac{1}{2},j}^{n}\right.\right.
       \end{equation*}
       \begin{equation*}
        \left.\left.+\frac{1}{R}\left(\delta_{x}^{2}u_{ij}^{n}-\delta_{x}^{2}u_{i-1,j}^{n}\right)\right]+O(k+k^{2}h^{-1})\right\}\left\{u^{n}_{ij}+
        \frac{k}{2}\left[-u_{ij}^{n}\delta_{x}u_{i+\frac{1}{2},j}^{n}+\frac{1}{R}\delta_{x}^{2}u_{ij}^{n}\right]+O(k^{2}+h)\right\}.
       \end{equation*}
       Expanding this and absorbing the term of second order into the infinitesimal term $O(k+k^{2}h^{-1})$, we obtain
       \begin{equation*}
       u_{ij}^{\overline{*}}\delta_{x}u_{i-\frac{1}{2},j}^{\overline{*}}=u_{ij}^{n}\delta_{x}u_{i-\frac{1}{2},j}^{2}+
       \frac{k}{2}\left[-u_{ij}^{n}\delta_{x}u_{i-\frac{1}{2},j}^{n}\delta_{x}u_{i+\frac{1}{2},j}^{n}
       +\frac{1}{R}\delta_{x}u_{i-\frac{1}{2},j}^{n}\delta_{x}^{2}u_{ij}^{n}-\frac{1}{h}(u_{ij}^{n})^{2}\delta_{x}u_{i+\frac{1}{2},j}^{n}+\right.
       \end{equation*}
       \begin{equation*}
        \left.\frac{1}{h}u_{ij}^{n}u_{i-1,j}^{n}\delta_{x}u_{i-\frac{1}{2},j}^{n}+\frac{1}{R}u_{ij}^{n}\delta_{x}^{2}(\delta_{x}u_{i-\frac{1}{2},j}^{n})
        \right]+O(k+k^{2}h^{-1}),
       \end{equation*}
       which can be rewritten as
       \begin{equation}\label{26}
       u_{ij}^{\overline{*}}\delta_{x}u_{i-\frac{1}{2},j}^{\overline{*}}=u_{ij}^{n}\delta_{x}u_{i-\frac{1}{2},j}^{n}+O(k+k^{2}h^{-1}),
       \end{equation}
       where we have also absorbed the first order term into the error term $O(k+k^{2}h^{-1}).$ Using approximation $(\ref{26}),$ it is too simple to observe that
       \begin{equation}\label{27}
       u_{ij}^{\overline{*}}\delta_{x}u_{i-\frac{1}{2},j}^{\overline{*}}+u_{ij}^{n}\delta_{x}u_{i+\frac{1}{2},j}^{n}=
       u_{ij}^{n}(\delta_{x}u_{i-\frac{1}{2},j}^{n}+\delta_{x}u_{i+\frac{1}{2},j}^{n})+O(k+k^{2}h^{-1}).
       \end{equation}
       But, it comes from the definition of the linear operator $\delta^{x}$ given by equation $(\ref{4})$ that $\delta^{x}u_{ij}^{n}=\frac{1}{2}(\delta_{x}u_{i-\frac{1}{2},j}^{n}+\delta_{x}u_{i+\frac{1}{2},j}^{n}).$ This fact, together with relation $(\ref{27})$ yield
       \begin{equation}\label{28}
       u_{ij}^{\overline{*}}\delta_{x}u_{i-\frac{1}{2},j}^{\overline{*}}+u_{ij}^{n}\delta_{x}u_{i+\frac{1}{2},j}^{n}=
       2u_{ij}^{n}\delta^{x}u_{ij}^{n}+O(k+k^{2}h^{-1}).
       \end{equation}
       Furthermore, plugging the first equation in $(\ref{12})$ together with the second equation in $(\ref{25}),$ direct calculations result in
       \begin{equation*}
        \delta_{x}^{2}u_{ij}^{\overline{*}}=\frac{u_{i+1,j}^{\overline{*}}-2u_{ij}^{\overline{*}}+u_{i-1,j}^{\overline{*}}}{h^{2}}=
        \frac{1}{h^{2}}\left\{u^{n}_{i+1,j}+\frac{k}{2}[-uu_{x}+\frac{1}{R}u_{xx}]_{i+1,j}^{n}-2u^{n}_{ij}
        -k[-uu_{x}+\frac{1}{R}u_{xx}]_{ij}^{n}\right.
       \end{equation*}
       \begin{equation*}
        \left.+u^{n}_{i-1,j}+\frac{k}{2}[-uu_{x}+\frac{1}{R}u_{xx}]_{i-1,j}^{n}\right\}+O(k^{2}h^{-2})=\delta_{x}^{2}u_{ij}^{n}
        +\frac{k}{2h^{2}}\left\{-u^{n}_{i+1,j}u^{n}_{x,i+1,j}+2u^{n}_{ij}u^{n}_{x,ij}-\right.
       \end{equation*}
       \begin{equation}\label{30}
        \left.u^{n}_{i-1,j}u^{n}_{x,i-1,j}+\frac{1}{R}\left(u^{n}_{2x,i+1,j}-2u^{n}_{2x,ij}+u^{n}_{2x,i-1,j}\right)\right\}+O(k^{2}h^{-2}).
       \end{equation}
        Expanding the Taylor series about $(x_{i},y_{j},t^{n})$ with mesh size $h$ using both forward and backward difference representations, it is not hard to see that
       \begin{equation*}
       u_{i+1,j}^{n}=u^{n}_{ij}+hu_{x,ij}^{n}+O(h^{2}),\text{\,}u_{x,i+1,j}^{n}=u_{x,ij}^{n}+hu_{2x,ij}^{n}+O(h^{2}),
       \text{\,}u_{2x,i+1,j}^{n}=u_{2x,ij}^{n}+hu_{3x,ij}^{n}+O(h^{2}),
       \end{equation*}
       \begin{equation}\label{31}
        u_{i-1,j}^{n}=u^{n}_{ij}-hu_{x,ij}^{n}+O(h^{2}),\text{\,}u_{x,i-1,j}^{n}=u_{x,ij}^{n}-hu_{2x,ij}^{n}+O(h^{2}),
       \text{\,}u_{2x,i-1,j}^{n}=u_{2x,ij}^{n}-hu_{3x,ij}^{n}+O(h^{2}).
       \end{equation}
       Using this, it is easy to see that
       \begin{equation}\label{32}
       u_{i+1,j}^{n}u_{x,i+1,j}^{n}=u_{ij}^{n}u_{x,ij}^{n}+h[u_{ij}^{n}u_{x,ij}^{n}+(u_{x,ij}^{n})^{2}]+O(h^{2}),
       \end{equation}
       \begin{equation}\label{32*}
        u_{i-1,j}^{n}u_{x,i-1,j}^{n}=u_{ij}^{n}u_{x,ij}^{n}-h[u_{ij}^{n}u_{x,ij}^{n}+(u_{x,ij}^{n})^{2}]+O(h^{2}).
       \end{equation}
       Substituting equations $(\ref{31}),$ $(\ref{32})$ and $(\ref{32*})$ into  $(\ref{30})$, and after simplification we obtain
       \begin{equation*}
        \frac{u_{i+1,j}^{\overline{*}}-2u_{ij}^{\overline{*}}+u_{i-1,j}^{\overline{*}}}{h^{2}}
        =\delta_{x}^{2}u_{ij}^{n}+\frac{k}{2h^{2}}\left\{O(h^{2})+O(h^{2})\right\}+O(k^{2}h^{-2}),
       \end{equation*}
       which can be rewritten as
       \begin{equation*}
        \delta_{x}^{2}u_{ij}^{\overline{*}}=\delta_{x}^{2}u_{ij}^{n}+O(k+k^{2}h^{-2}).
       \end{equation*}
       Thus,
       \begin{equation*}
        \frac{1}{R}\left(\delta_{x}^{2}u_{ij}^{\overline{*}}+\delta_{x}^{2}u_{ij}^{n}\right)=\frac{2}{R}\delta_{x}^{2}u_{ij}^{n}+O(k+k^{2}h^{-2}).
       \end{equation*}
       This ends the proof of Lemma $\ref{l1}.$
       \end{proof}
       Now, using Lemma $\ref{l1},$ we are ready to give a full description of the three-level explicit time-split MacCormack approach applied to the two-dimensional time-dependent nonlinear coupled Burgers' equations $(\ref{1})$-$(\ref{3})$ and to provide the convergence rate of the algorithm.

       Combining equations $(\ref{19}),$ $(\ref{21})$ and $(\ref{23}),$ (respectively, $(\ref{20}),$ $(\ref{22})$ and $(\ref{24}),$), direct calculations give
       \begin{equation}\label{33}
        \frac{u_{ij}^{\overline{*}}+u_{ij}^{\overline{\overline{*}}}}{2}=u^{n}_{ij}+\frac{k}{2}\left[-u_{ij}^{n}\delta^{x}u_{ij}^{n}
        +\frac{1}{R}\delta^{2}_{x}u_{ij}^{n}\right]+O(k^{2}+k^{3}h^{-1}+k^{3}h^{-2}),
       \end{equation}
       and
       \begin{equation}\label{34}
        \frac{v_{ij}^{\overline{*}}+v_{ij}^{\overline{\overline{*}}}}{2}=v^{n}_{ij}+\frac{k}{2}\left[-u_{ij}^{n}\delta^{x}v_{ij}^{n}
        +\frac{1}{R}\delta^{2}_{x}v_{ij}^{n}\right]+O(k^{2}+k^{3}h^{-1}+k^{3}h^{-2}).
       \end{equation}
        For low Reynolds numbers, the time step restriction $(\ref{11**})$ is dominated by the inequality $2R^{-1}k\leq h^{2},$ so it is easy to see that
        \begin{equation*}
        k^{3}h^{-1}\leq \frac{R^{3}}{8}h^{5} \text{\,\,\,and\,\,\,} k^{3}h^{-2}\leq \frac{R^{3}}{8}h^{4}.
        \end{equation*}
        Utilizing this, equations $(\ref{33})$ and $(\ref{34})$ become
        \begin{equation}\label{35}
        \frac{u_{ij}^{\overline{*}}+u_{ij}^{\overline{\overline{*}}}}{2}=u^{n}_{ij}+\frac{k}{2}\left[-u_{ij}^{n}\delta^{x}u_{ij}^{n}
        +\frac{1}{R}\delta^{2}_{x}u_{ij}^{n}\right]+O(k^{2}+h^{4}),
       \end{equation}
       and
       \begin{equation}\label{36}
        \frac{v_{ij}^{\overline{*}}+v_{ij}^{\overline{\overline{*}}}}{2}=v^{n}_{ij}+\frac{k}{2}\left[-u_{ij}^{n}\delta^{x}v_{ij}^{n}
        +\frac{1}{R}\delta^{2}_{x}v_{ij}^{n}\right]+O(k^{2}+h^{4}).
       \end{equation}
        In a like manner, for high Reynolds numbers, the time step restriction $(\ref{11**})$ is dominated by estimate $k^{\frac{3}{4}}\leq h,$ so it holds
        \begin{equation*}
        k^{3}h^{-1}\leq h^{3} \text{\,\,\,and\,\,\,} k^{3}h^{-2}\leq h^{2}.
        \end{equation*}
        Thus, equations $(\ref{33})$ and $(\ref{34})$ imply
        \begin{equation}\label{35*}
        \frac{u_{ij}^{\overline{*}}+u_{ij}^{\overline{\overline{*}}}}{2}=u^{n}_{ij}+\frac{k}{2}\left[-u_{ij}^{n}\delta^{x}u_{ij}^{n}
        +\frac{1}{R}\delta^{2}_{x}u_{ij}^{n}\right]+O(k^{2}+h^{2}),
       \end{equation}
       and
       \begin{equation}\label{36*}
        \frac{v_{ij}^{\overline{*}}+v_{ij}^{\overline{\overline{*}}}}{2}=v^{n}_{ij}+\frac{k}{2}\left[-u_{ij}^{n}\delta^{x}v_{ij}^{n}
        +\frac{1}{R}\delta^{2}_{x}v_{ij}^{n}\right]+O(k^{2}+h^{2}).
       \end{equation}
        Analogously, to construct the nonlinear operator $L_{y}(k)$, we consider the following one-dimensional system of equations
       \begin{equation*}
        u_{t}+vu_{y}=\frac{1}{R}u_{yy}\text{\,\,and\,\,}v_{t}+vv_{x}=\frac{1}{R}v_{yy}.
       \end{equation*}
          Applying the Taylor series expansion about $(x_{i},y_{i},t^{*})$ (where $t^{*}\in(t^{n},t^{n+1}),$ is the starting time used at the next phase in a time-split MacCormack scheme) with time step $k$ and mesh size $h$ using both forward and backward difference representations, it is not difficult to show that
         \begin{equation}\label{37}
        \frac{u_{ij}^{\overline{**}}+u_{ij}^{\overline{\overline{**}}}}{2}=u^{*}_{ij}+k\left[-v_{ij}^{*}\delta^{y}u_{ij}^{*}
        +\frac{1}{R}\delta^{2}_{y}u_{ij}^{*}\right]+O(k^{2}+h^{4}),
       \end{equation}
       \begin{equation}\label{38}
        \frac{v_{ij}^{\overline{**}}+v_{ij}^{\overline{\overline{**}}}}{2}=v^{*}_{ij}+k\left[-v_{ij}^{*}\delta^{y}v_{ij}^{*}
        +\frac{1}{R}\delta^{2}_{y}v_{ij}^{*}\right]+O(k^{2}+h^{4}),
       \end{equation}
         for low Reynolds numbers, and for large Reynolds numbers
         \begin{equation}\label{37*}
        \frac{u_{ij}^{\overline{**}}+u_{ij}^{\overline{\overline{**}}}}{2}=u^{*}_{ij}+k\left[-v_{ij}^{*}\delta^{y}u_{ij}^{*}
        +\frac{1}{R}\delta^{2}_{y}u_{ij}^{*}\right]+O(k^{2}+h^{2}),
       \end{equation}
       \begin{equation}\label{38*}
        \frac{v_{ij}^{\overline{**}}+v_{ij}^{\overline{\overline{**}}}}{2}=v^{*}_{ij}+k\left[-v_{ij}^{*}\delta^{y}v_{ij}^{*}
        +\frac{1}{R}\delta^{2}_{y}v_{ij}^{*}\right]+O(k^{2}+h^{2}).
       \end{equation}
        Finally, considering the one-dimensional equations system
         \begin{equation*}
        u_{t}=-vu_{x}+\frac{1}{R}u_{xx}\text{\,\,and\,\,}v_{t}=-uv_{x}+\frac{1}{R}v_{xx},
       \end{equation*}
          expanding the Taylor series about $(x_{i},y_{j},t^{**})$ (where $t^{**}$ represents the time used at the last phase in the time-split MacCormack procedure) at both predictor and corrector phases with time step $k/2$ and mesh size $h,$ using forward and backward difference formulations, it is not hard to show that when the Reynolds numbers are small
         \begin{equation}\label{39}
        \frac{u_{ij}^{\overline{n+1}}+u_{ij}^{\overline{\overline{n+1}}}}{2}=u^{**}_{ij}+\frac{k}{2}\left[-u_{ij}^{**}\delta^{x}u_{ij}^{**}
        +\frac{1}{R}\delta^{2}_{x}u_{ij}^{**}\right]+O(k^{2}+h^{4}),
       \end{equation}
       \begin{equation}\label{40}
        \frac{v_{ij}^{\overline{n+1}}+v_{ij}^{\overline{\overline{n+1}}}}{2}=v^{**}_{ij}+\frac{k}{2}\left[-u_{ij}^{**}\delta^{x}v_{ij}^{**}
        +\frac{1}{R}\delta^{2}_{x}v_{ij}^{**}\right]+O(k^{2}+h^{4}),
       \end{equation}
        and for high Reynolds numbers
        \begin{equation}\label{39*}
        \frac{u_{ij}^{\overline{n+1}}+u_{ij}^{\overline{\overline{n+1}}}}{2}=u^{**}_{ij}+\frac{k}{2}\left[-u_{ij}^{**}\delta^{x}u_{ij}^{**}
        +\frac{1}{R}\delta^{2}_{x}u_{ij}^{**}\right]+O(k^{2}+h^{2}),
       \end{equation}
       \begin{equation}\label{40*}
        \frac{v_{ij}^{\overline{n+1}}+v_{ij}^{\overline{\overline{n+1}}}}{2}=v^{**}_{ij}+\frac{k}{2}\left[-u_{ij}^{**}\delta^{x}v_{ij}^{**}
        +\frac{1}{R}\delta^{2}_{x}v_{ij}^{**}\right]+O(k^{2}+h^{2}).
       \end{equation}
         In order to develop the three-level explicit time-split MacCormack method for solving the two-dimensional evolutionary nonlinear coupled Burgers' equations $(\ref{1})$ with initial and boundary conditions $(\ref{2})$-$(\ref{3}),$ we must neglect the infinitesimal terms $O(k^{2}+h^{4})$ and $O(k^{2}+h^{2})$ in equations $(\ref{35})$-$(\ref{40*})$. In addition, we introduce the terms $u_{ij}^{*},$ $u_{ij}^{**}$, $u^{n+1}_{ij}$, $v_{ij}^{*},$ $v_{ij}^{**}$ and $v^{n+1}_{ij}$ which are defined as follows,
       \begin{equation*}
        u_{ij}^{*}=\frac{u_{ij}^{\overline{*}}+u_{ij}^{\overline{\overline{*}}}}{2},\text{\,\,\,}u_{ij}^{**}=
        \frac{u_{ij}^{\overline{**}}+u_{ij}^{\overline{\overline{**}}}}{2}\text{\,\,\,\,\,\,and\,\,\,\,\,\,}u_{ij}^{n+1}=
        \frac{u_{ij}^{\overline{n+1}}+u_{ij}^{\overline{\overline{n+1}}}}{2},
       \end{equation*}
        \begin{equation*}
        v_{ij}^{*}=\frac{v_{ij}^{\overline{*}}+v_{ij}^{\overline{\overline{*}}}}{2},\text{\,\,\,}v_{ij}^{**}=
        \frac{v_{ij}^{\overline{**}}+v_{ij}^{\overline{\overline{**}}}}{2}\text{\,\,\,\,\,\,and\,\,\,\,\,\,}v_{ij}^{n+1}=
        \frac{v_{ij}^{\overline{n+1}}+v_{ij}^{\overline{\overline{n+1}}}}{2}.
       \end{equation*}
       Thus, systems of equations
       \begin{equation}\label{44}
       \begin{pmatrix}
                    u_{ij}^{*} \\
                    v_{ij}^{*} \\
                  \end{pmatrix}=L_{x}(k/2)\begin{pmatrix}
                                                     u_{ij}^{n} \\
                                                     v_{ij}^{n} \\
                                                   \end{pmatrix},\text{\,\,}\begin{pmatrix}
                    u_{ij}^{**} \\
                    v_{ij}^{**} \\
                  \end{pmatrix}=L_{y}(k)\begin{pmatrix}
                                                     u_{ij}^{*} \\
                                                     v_{ij}^{*} \\
                                                   \end{pmatrix}
                                                   ,\text{\,\,}\begin{pmatrix}
                    u_{ij}^{n+1} \\
                    v_{ij}^{n+1} \\
                  \end{pmatrix}=L_{x}(k/2)\begin{pmatrix}
                                                     u_{ij}^{**} \\
                                                     v_{ij}^{**} \\
                                                   \end{pmatrix},
       \end{equation}
       are by definition equivalent to
       \begin{equation}\label{41}
        u^{*}=u^{n}_{ij}+\frac{k}{2}\left[-u_{ij}^{n}\delta^{x}u_{ij}^{n}+\frac{1}{R}\delta^{2}_{x}u_{ij}^{n}\right],\text{\,\,\,}
        v^{*}=v^{n}_{ij}+\frac{k}{2}\left[-u_{ij}^{n}\delta^{x}v_{ij}^{n}+\frac{1}{R}\delta^{2}_{x}v_{ij}^{n}\right],
       \end{equation}
        \begin{equation}\label{42}
        u^{**}=u^{*}_{ij}+k\left[-v_{ij}^{*}\delta^{y}u_{ij}^{*}+\frac{1}{R}\delta^{2}_{y}u_{ij}^{*}\right],\text{\,\,\,}
        v^{**}=v^{*}_{ij}+k\left[-v_{ij}^{*}\delta^{y}v_{ij}^{*}+\frac{1}{R}\delta^{2}_{y}v_{ij}^{*}\right],
       \end{equation}
        \begin{equation}\label{43}
        u^{n+1}=u^{**}_{ij}+\frac{k}{2}\left[-u_{ij}^{**}\delta^{x}u_{ij}^{**}+\frac{1}{R}\delta^{2}_{x}u_{ij}^{**}\right],\text{\,\,\,}
        v^{n+1}=v^{**}_{ij}+\frac{k}{2}\left[-u_{ij}^{**}\delta^{x}v_{ij}^{**}+\frac{1}{R}\delta^{2}_{x}v_{ij}^{**}\right].
       \end{equation}
       Using relation $(\ref{44}),$ it is easy to see that the nonlinear operator $L_{x}(k/2)L_{y}(k)L_{x}(k/2)$ is symmetric and satisfies
        \begin{equation}\label{47}
        \begin{pmatrix}
                    u_{ij}^{n+1} \\
                    v_{ij}^{n+1} \\
                  \end{pmatrix}=L_{x}(k/2)L_{y}(k)L_{x}(k/2)\begin{pmatrix}
                                                     u_{ij}^{n} \\
                                                     v_{ij}^{n} \\
                                                   \end{pmatrix}.
        \end{equation}

        This fact, together with relations $(\ref{35})$-$(\ref{40*})$ suggest that the three-level time-split explicit MacCormack approach applied to the parabolic system of nonlinear partial differential equations $(\ref{1})$-$(\ref{3})$ is an explicit predictor-corrector scheme, second order accurate in time and fourth order convergent in space for small Reynolds numbers. Furthermore, for high Reynolds numbers the method is second order convergent in both time and space. We confirm this convergence rate in section $\ref{sec3}$ by performing a large set of numerical evidences. Finally, from the definition of the linear operators $"\delta^{x}"$, $"\delta^{y}"$, $"\delta_{x}^{2}"$ and $"\delta_{y}^{2}"$ given in relation $(\ref{4})$, equations $(\ref{41})$-$(\ref{43})$ are equivalent to, for $n=0,1,...,N-1;$
       \begin{equation}\label{48}
        u_{ij}^{*}=u^{n}_{ij}+\frac{k}{2}\left[-u_{ij}^{n}\frac{u_{i+1,j}^{n}-u_{i-1,j}^{n}}{2h}+\frac{1}{R}
        \frac{u_{i+1,j}^{n}-2u_{ij}^{n}+u_{i-1,j}^{n}}{h^{2}}\right],\text{\,\,}i=1,2,...,M-1,\text{\,\,} j=0,1,...,M,
       \end{equation}
       \begin{equation}\label{49}
        v_{ij}^{*}=v^{n}_{ij}+\frac{k}{2}\left[-u_{ij}^{n}\frac{v_{i+1,j}^{n}-v_{i-1,j}^{n}}{2h}+\frac{1}{R}
        \frac{v_{i+1,j}^{n}-2v_{ij}^{n}+v_{i-1,j}^{n}}{h^{2}}\right],\text{\,\,}i=1,2,...,M-1,\text{\,\,} j=0,1,...,M,
       \end{equation}
       \begin{equation}\label{50}
        u_{ij}^{**}=u^{*}_{ij}+k\left[-v_{ij}^{*}\frac{u_{i,j+1}^{*}-u_{i,j-1}^{*}}{2h}+\frac{1}{R}
        \frac{u_{i,j+1}^{*}-2u_{ij}^{*}+u_{i,j-1}^{*}}{h^{2}}\right],\text{\,\,}i=0,1,...,M,\text{\,\,} j=1,2,...,M-1,
       \end{equation}
       \begin{equation}\label{51}
        v_{ij}^{**}=v^{*}_{ij}+k\left[-v_{ij}^{*}\frac{v_{i,j+1}^{*}-v_{i,j-1}^{*}}{2h}+\frac{1}{R}
        \frac{v_{i,j+1}^{*}-2v_{ij}^{*}+v_{i,j-1}^{*}}{h^{2}}\right],\text{\,\,}i=0,1,...,M,\text{\,\,} j=1,2,...,M-1,
       \end{equation}
        \begin{equation}\label{52}
        u_{ij}^{n+1}=u^{**}_{ij}+\frac{k}{2}\left[-u_{ij}^{**}\frac{u_{i+1,j}^{**}-u_{i-1,j}^{**}}{2h}+\frac{1}{R}
        \frac{u_{i+1,j}^{**}-2u_{ij}^{**}+u_{i-1,j}^{**}}{h^{2}}\right],\text{\,\,}i=1,2,...,M-1,\text{\,\,} j=0,1,...,M,
       \end{equation}
       \begin{equation}\label{53}
        v_{ij}^{n+1}=v^{**}_{ij}+\frac{k}{2}\left[-u_{ij}^{**}\frac{v_{i+1,j}^{**}-v_{i-1,j}^{**}}{2h}+\frac{1}{R}
        \frac{v_{i+1,j}^{**}-2v_{ij}^{**}+v_{i-1,j}^{**}}{h^{2}}\right],\text{\,\,}i=1,2,...,M-1,\text{\,\,} j=0,1,...,M,
       \end{equation}
        with initial and boundary conditions,
        \begin{equation*}
        u_{ij}^{0}=u_{0}(x_{i},y_{j}),\text{\,}u_{i0}^{n}=\varphi^{n}_{1,i0},\text{\,}u_{iM}^{n}=\varphi^{n}_{1,iM},
        \text{\,}u_{0j}^{n}=\varphi^{n}_{1,0j},\text{\,}u_{Mj}^{n}=\varphi^{n}_{1,Mj},\text{\,\,}u_{0j}^{*}=\varphi^{n+1}_{1,0j},\text{\,}u_{Mj}^{*}=
        \varphi^{n+1}_{1,Mj}, \text{\,}u_{j0}^{*}=\varphi^{n+1}_{1,j0},
        \end{equation*}
        \begin{equation*}
        u_{jM}^{*}=\varphi^{n+1}_{1,jM},\text{\,\,\,}u_{0j}^{**}=\varphi^{n+1}_{1,0j},\text{\,}u_{Mj}^{**}=\varphi^{n+1}_{1,Mj},\text{\,}
        u_{j0}^{**}=\varphi^{n+1}_{1,j0},\text{\,}u_{jM}^{**}=\varphi^{n+1}_{1,jM},\text{\,}u_{i0}^{N}=\varphi^{N}_{1,i0},\text{\,}u_{iM}^{N}=
        \varphi^{N}_{1,iM},
        \end{equation*}
        \begin{equation}\label{45}
         u_{0j}^{N}=\varphi^{N}_{1,0j},\text{\,}u_{Mj}^{N}=\varphi^{N}_{1,Mj},\text{\,\,\,for\,\,\,} i,j=0,1,...,M.
        \end{equation}
        \begin{equation*}
        v_{ij}^{0}=v_{0}(x_{i},y_{j}),\text{\,}v_{i0}^{n}=\varphi^{n}_{2,i0},\text{\,}v_{iM}^{n}=\varphi^{n}_{2,iM},
        \text{\,}v_{0j}^{n}=\varphi^{n}_{2,0j},\text{\,}v_{Mj}^{n}=\varphi^{n}_{2,Mj},\text{\,\,}v_{0j}^{*}=\varphi^{n+1}_{2,0j},\text{\,}v_{Mj}^{*}=
        \varphi^{n+1}_{2,Mj}, \text{\,}v_{j0}^{*}=\varphi^{n+1}_{2,j0},
        \end{equation*}
        \begin{equation*}
        v_{jM}^{*}=\varphi^{n+1}_{2,jM},\text{\,\,\,}v_{0j}^{**}=\varphi^{n+1}_{2,0j},\text{\,}v_{Mj}^{**}=\varphi^{n+1}_{2,Mj},\text{\,}
        v_{j0}^{**}=\varphi^{n+1}_{2,j0},\text{\,}v_{jM}^{**}=\varphi^{n+1}_{2,jM},\text{\,}v_{i0}^{N}=\varphi^{N}_{2,i0},\text{\,}v_{iM}^{N}=
        \varphi^{N}_{2,iM},
        \end{equation*}
        \begin{equation}\label{54}
         v_{0j}^{N}=\varphi^{N}_{2,0j},\text{\,}v_{Mj}^{N}=\varphi^{N}_{2,Mj},\text{\,\,\,for\,\,\,} i,j=0,1,...,M.
        \end{equation}
        It is important to mention that equations $(\ref{48})$-$(\ref{54})$ denote a detailed description of the three-level explicit time-split MacCormack algorithm in a numerical solution of the initial-boundary value problem $(\ref{1})$-$(\ref{3}).$

         \section{Numerical experiments and Convergence rate}\label{sec3}
           In this section we present numerical evidences for the two-dimensional time dependent nonlinear coupled Burgers equations $(\ref{1})$-$(\ref{3}).$ The exact solutions used in our experiments are taken in \cite{cf}. For low Reynolds numbers, the test suggests that the proposed scheme is second order accurate in time and fourth order convergent in space while for low Mach numbers (i.e., high Reynolds numbers) the numerical examples show that the algorithm is second order convergent in both time and space. These observations confirm the theoretical analysis (see section $\ref{sec2}$, Page $7$, last paragraph) and the predicted results provided in the literature (for instance, see \cite{apt}, page 632). The convergence rate is obtained by listing in Tables $1$-$4$ the errors between the approximate solution and the analytical ones with different values of time step $k$ and grid spacing $h$ satisfying $k=\frac{R}{2}h^{2}$ and $k\leq h^{\frac{4}{3}}.$ Furthermore, we look at the error estimates of the method for the parameter $T=1$ and the Reynolds numbers $R\in\{2,64\}$.\\

          In the numerical tests, we assume that the mesh size $h\in\{\frac{1}{2},\frac{1}{2^{2}},\frac{1}{2^{3}}, \frac{1}{2^{4}},
        \frac{1}{2^{5}},\frac{1}{2^{6}},\frac{1}{2^{7}},\frac{1}{2^{8}}\}$ and time step $k\in\{\frac{1}{2^{2}},\frac{1}{2^{3}},\frac{1}{2^{4}},
        \frac{1}{2^{5}},\frac{1}{2^{6}},\frac{1}{2^{7}},\frac{1}{2^{8}},\frac{1}{2^{9}},\frac{1}{2^{10}}\frac{1}{2^{11}}\}$. Furthermore, we compute the error estimates: $\||E(\phi)|\|_{L^{2}(0,T;L^{2})},$ $\||E(\phi)|\|_{L^{\infty}(0,T;L^{2})}$ and $\||E(\phi)|\|_{L^{1}(0,T;L^{2})}$ (for $\phi=u,v$), associated with the three-level time-split scheme to demonstrate the efficiency and effectiveness of our method in two-dimensional case (stable, second order convergent in time and fourth order accurate in space). We plot the exact solution, computed solution and errors versus $n.$ It comes from this analysis that the three-level time-split MacCormack method is more fast and efficient than a wide range of numerical schemes widely studied in the literature. Finally, Tables $1$-$4$ suggest that the error terms $O(k^{\beta})+O(h^{\theta})$ are dominated by the h-terms $O(h^{\theta})$ ($k$-terms $O(k^{\beta})$). Thus, the numbers $\theta$ (respectively $\beta$) can be used to estimate the corresponding convergence rate with respect to $h$ (respectively, $k$). Define the norms for the numerical solution $\phi,$ the exact one $\overline{\phi},$ and the errors $E(\phi),$ as follows
         \begin{equation*}
         \||\phi|\|_{L^{2}(0,T;L^{2})}=\left[k\underset{n=0}{\overset{N}\sum}\|\phi^{n}\|_{L_{f}^{2}}^{2}\right]^{\frac{1}{2}};
         \text{\,\,}\||\overline{\phi}|\|_{L^{2}(0,T;L^{2})}=\left[k\underset{n=0}{\overset{N}\sum}\|\overline{\phi}^{n}\|_{L_{f}^{2}}^{2}\right]^{\frac{1}{2}};
         \end{equation*}
         \begin{equation*}
           \||E(\phi)|\|_{L^{2}(0,T;L^{2})}=\left[k\underset{n=0}{\overset{N}\sum}\|\phi^{n}-\overline{\phi}^{n}\|_{L_{f}^{2}}^{2}\right]^{\frac{1}{2}};
           \text{\,\,}\||E(\phi)|\|_{L^{1}(0,T;L^{2})}=k\underset{n=0}{\overset{N}\sum}\|\phi^{n}-\overline{\phi}^{n}\|_{L_{f}^{2}};
         \end{equation*}
         and
         \begin{equation*}
            \||E(\phi)|\|_{L^{\infty}(0,T;L^{2})}=\underset{0\leq n\leq N}{\max}\|\phi^{n}-\overline{\phi}^{n}\|_{L_{f}^{2}}.
         \end{equation*}

          $\bullet$ \textbf{Test.} Consider $\Omega$ be the unit square $(0,1)^{2}$ and $T=1,$ be the final time. The examples compare the numerical solutions with the exact ones to verify whether the proposed method leads to high accuracy. We assume that the Reynolds number $R\in\{2,64\},$ such that the exact solutions $\overline{u}$ and $\overline{v}$ taken in \cite{cf} are given by
       \begin{equation*}
         \overline{u}(x,y,t)=\frac{1}{4}\left[3-\left(1+\exp\left(R(-t-4x+4y)/32\right)\right)^{-1}\right],
       \end{equation*}
       \begin{equation*}
         \overline{v}(x,y,t)=\frac{1}{4}\left[3+\left(1+\exp\left(R(-t-4x+4y)/32\right)\right)^{-1}\right].
       \end{equation*}
       The initial and boundary conditions are determined by this solution. We recall that the mesh size and time step: $h\in\{\frac{1}{2},\frac{1}{2^{2}},
       \frac{1}{2^{3}},\frac{1}{2^{4}},\frac{1}{2^{5}},\frac{1}{2^{6}},\frac{1}{2^{7}},\frac{1}{2^{8}}\}$ and $k\in\{\frac{1}{2^{2}},\frac{1}{2^{3}},\frac{1}{2^{4}},\frac{1}{2^{5}},
       \frac{1}{2^{6}},\frac{1}{2^{7}},\frac{1}{2^{8}},\frac{1}{2^{9}},\frac{1}{2^{10}},\frac{1}{2^{11}}\}.$\\

           \textbf{Tables 1,2.} Analysis of convergence rate $O(h^{\theta}+\Delta t^{\beta})$ for the three-level time-split MacCormack under the time step restriction $(\ref{11**})$, that is, $\frac{2R^{-1}k}{h^{2}}\leq1$, for low Reynolds number (for example, $R=2$), varying time step $k=\Delta t$ and mesh grid $h=\Delta x=\Delta y$.\\

            \textbf{Table 1. $k=\frac{R}{2}h^{2}=h^{2}$.}
           $$\begin{tabular}{|c|c|c|c|c|c|c|}
            \hline
            $h$ & $ \||E(u)|\|_{L^{2}}$ & $\||E(v)|\|_{L^{2}}$ & $\||E(u)|\|_{L^{\infty}}$ & $\||E(v)|\|_{L^{\infty}}$ & $ \||E(u)|\|_{L^{1}}$ & $\||E(v)|\|_{L^{1}}$ \\
            \hline
            $2^{-1}$ & $7.391\times10^{-4}$ & $7.391\times10^{-4}$ & $7.926\times10^{-4}$ & $7.926\times10^{-4}$& $7.316\times10^{-4}$ &$7.316\times10^{-4}$\\
            \hline
            $2^{-2}$ & $4.285\times10^{-4}$ & $4.285\times10^{-4}$ & $4.537\times10^{-4}$ & $4.537\times10^{-4}$ & $4.248\times10^{-4}$ &$4.248\times10^{-4}$\\
            \hline
            $2^{-3}$ & $3.671\times10^{-4}$ & $3.671\times10^{-4}$ & $3.957\times10^{-4}$ & $3.957\times10^{-4}$ & $3.594\times10^{-4}$ & $3.594\times10^{-4}$\\
            \hline
            $2^{-4}$ & $3.647\times10^{-4}$ & $3.647\times10^{-4}$ & $3.938\times10^{-4}$ & $3.938\times10^{-4}$ & $3.566\times10^{-4}$
            & $3.566\times10^{-4}$\\
            \hline
          \end{tabular}$$
               \text{\,}\\
              \textbf{ Table 2. $k=h$.}
          $$\begin{tabular}{|c|c|c|c|c|c|c|}
            \hline
            $h$ & $ \||E(u)|\|_{L^{2}}$ & $\||E(v)|\|_{L^{2}}$ & $\||E(u)|\|_{L^{\infty}}$ & $\||E(v)|\|_{L^{\infty}}$ & $ \||E(u)|\|_{L^{1}}$ & $\||E(v)|\|_{L^{1}}$ \\
            \hline
            $2^{-1}$ & $0.0027$ & $0.0027$ & $0.0032$ & $0.0032$ & $0.0027$ & $0.0027$ \\
            \hline
            $2^{-2}$ & $18.9821$ & $18.9821$ & $37.9579$ & $37.9579$ & $9.6690$ & $9.6690$ \\
            \hline
            $2^{-3}$  & NaN & NaN & Inf & Inf & NaN & NaN \\
            \hline
          \end{tabular}$$
          \textbf{Table 1} and \textbf{Figure} $\ref{fig1}$ suggest that under time step restriction $(\ref{11**})$, the proposed numerical scheme is stable, second order convergent in time and fourth order accurate in space, while \textbf{Table 2} and Figures $\ref{fig2}$-$\ref{fig3}$ indicate that when the time step constraint $(\ref{11**})$ is not satisfied, the considered algorithm is neither stable nor convergent.\\

           \textbf{Tables 3,4.} Convergence rate $O(h^{\theta}+\Delta t^{\beta})$ for the three-level time-split MacCormack under time step restriction $(\ref{44})$ (that is, $\frac{k^{\frac{3}{4}}}{h}\leq1$), for high Reynolds numbers (for instance, $R=64$), with varying time step $k=\Delta t$ and mesh grid $h=\Delta x=\Delta y$.\\

            \textbf{Table 3. $k=\frac{1}{4}h\leq h^{\frac{4}{3}}$.}
           $$\begin{tabular}{|c|c|c|c|c|c|c|}
            \hline
            $h$ & $ \||E(u)|\|_{L^{2}}$ & $\||E(v)|\|_{L^{2}}$ & $\||E(u)|\|_{L^{\infty}}$ & $\||E(v)|\|_{L^{\infty}}$ & $ \||E(u)|\|_{L^{1}}$ & $\||E(v)|\|_{L^{1}}$ \\
            \hline
            $2^{-3}$ & $3.9500\times10^{-2}$ & $3.9500\times10^{-2}$ & $5.8900\times10^{-2}$ & $5.8900\times10^{-2}$& $3.5500\times10^{-2}$ &$3.5500\times10^{-2}$\\
            \hline
            $2^{-4}$ & $3.3500\times10^{-2}$ & $3.3500\times10^{-2}$ & $4.6400\times10^{-2}$ & $4.6400\times10^{-2}$& $3.0400\times10^{-2}$ &$3.0400\times10^{-2}$\\
            \hline
            $2^{-5}$ & $3.2200\times10^{-2}$ & $3.2200\times10^{-2}$ & $4.3800\times10^{-2}$ & $4.3800\times10^{-2}$& $2.9400\times10^{-2}$ &$2.9400\times10^{-2}$\\
            \hline
            $2^{-6}$ & $3.1800\times10^{-2}$ & $3.1800\times10^{-2}$ & $4.3000\times10^{-2}$ & $4.3000\times10^{-2}$& $2.9000\times10^{-2}$ &$2.9000\times10^{-2}$\\
            \hline
            $2^{-7}$ & $3.1600\times10^{-2}$ & $3.1600\times10^{-2}$ & $4.2700\times10^{-2}$ & $4.2700\times10^{-2}$& $2.8800\times10^{-2}$ &$2.8800\times10^{-2}$\\
            \hline
          \end{tabular}$$
               \text{\,}\\
              \textbf{Table 4. $k=h>h^{\frac{4}{3}}$.}
          $$\begin{tabular}{|c|c|c|c|c|c|c|}
            \hline
            $k$ & $ \||E(u)|\|_{L^{2}}$ & $\||E(v)|\|_{L^{2}}$ & $\||E(u)|\|_{L^{\infty}}$ & $\||E(v)|\|_{L^{\infty}}$ & $ \||E(u)|\|_{L^{1}}$ & $\||E(v)|\|_{L^{1}}$ \\
            \hline
            $2^{-3}$ & $4.9400\times10^{-2}$ & $4.9400\times10^{-2}$ & $7.6100\times10^{-2}$ & $7.6100\times10^{-2}$& $4.4200\times10^{-2}$ &$4.4200\times10^{-2}$\\
            \hline
            $2^{-4}$ & $3.8600\times10^{-2}$ & $3.8600\times10^{-2}$ & $5.3900\times10^{-2}$ & $5.3900\times10^{-2}$& $3.5200\times10^{-2}$ &$3.5200\times10^{-2}$\\
            \hline
            $2^{-5}$ & $3.4200\times10^{-2}$ & $3.4200\times10^{-2}$ & $4.6400\times10^{-2}$ & $4.6400\times10^{-2}$& $3.1300\times10^{-2}$ &$3.1300\times10^{-2}$\\
            \hline
            $2^{-6}$ & NaN  & NaN  & Inf & Inf & NaN  & NaN \\
            \hline
          \end{tabular}$$

          \textbf{Table 3} and Figure $\ref{fig4}$ indicate that for high Reynolds numbers and under time step restriction $(\ref{11**})$, the considered method is stable, second order convergent in both time and space, while \textbf{Table 4} shows that the scheme is neither stable nor convergent for smallest mesh size $h$, whenever the time step limitation $(\ref{11**})$ is not accomplished.

         \section{General conclusion and future works}\label{sec4}
             In this paper, we have discussed the convergence rate of the three-level explicit time-split MacCormack algorithm for solving the two-dimensional time dependent nonlinear coupled Burgers' equations $(\ref{1})$ subject to the initial and boundary conditions $(\ref{2})$-$(\ref{3})$. For low Reynolds numbers, the theoretical results have shown that the proposed algorithm is stable, second order convergent in time and fourth order accurate in space, while for large Reynolds numbers, the analysis has demonstrated that the considered method is second order convergent in time and space. All this analysis has been done under the time step requirement $(\ref{11**})$. The theoretical study is confirmed by a wide set of numerical experiments (Figures $\ref{fig1}$-$\ref{fig4}$ and \textbf{Tables 1-4}). The numerical experiments also show that our method is: (a) more fast and efficient than a broad range of numerical schemes for solving the initial-boundary value problem $(\ref{1})$-$(\ref{3})$; (b) fast and robust tools for the integration of general systems of hyperbolic/parabolic PDEs. Unfortunately, we have observed from the numerical examples that the time-split MacCormack method is not too much efficient for solving high Reynolds number flows where the viscous region is very thin. The mesh size must be highly refined in order to accurately resolve the viscous regions. The small grid spacing leads to very small time steps and subsequently long computing times. Indeed, in the coarse-grid region, the three-level explicit time-split MacCormack approach can be applied, while in the fine-grid region, the following sequence of nonlinear operators can be used
             \begin{equation*}
              \begin{pmatrix}
                    u_{ij}^{n+1} \\
                    v_{ij}^{n+1} \\
                  \end{pmatrix}=\left[L_{x}\left(\frac{k}{2m}\right)L_{y}\left(\frac{k}{m}\right)L_{x}
                  \left(\frac{k}{2m}\right)\right]^{m}\begin{pmatrix}
                    u_{ij}^{n} \\
                    v_{ij}^{n} \\
                  \end{pmatrix},
                \end{equation*}
              where $m$ is the smallest integer satisfying inequality: $\max\left\{\frac{2k}{mRh^{2}},\frac{k^{\frac{3}{4}}}{m^{\frac{3}{4}}h}\right\}\leq1.$
             Our future investigations will consider the last formula together with the three-level time-split MacCormack technique in the numerical solutions of two-dimensional unsteady viscous coupled Burgers' equations $(\ref{1})$-$(\ref{3})$.\\

            \textbf{Acknowledgment.} The author would like to thank the deanship of scientific research of Imam Muhammad Ibn Saud Islamic University (IMSIU) to financially support this work under the Grant No. 331203.

           \begin{figure}
         \begin{center}
          Stability and convergence of a three-level time-split MacCormack method: $a=\mu=1$.
          \begin{tabular}{cc}
              \psfig{file=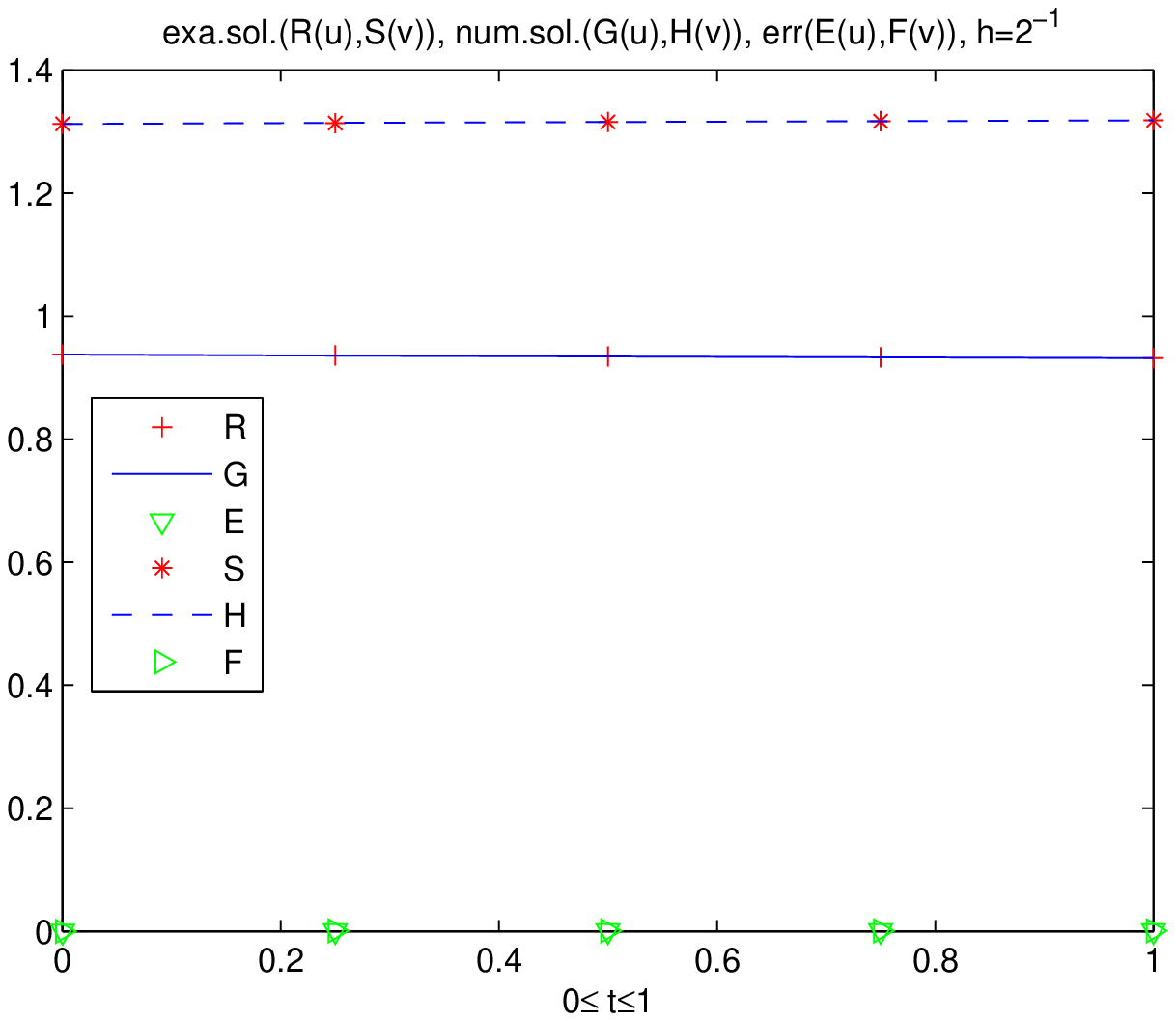,width=7cm} & \psfig{file=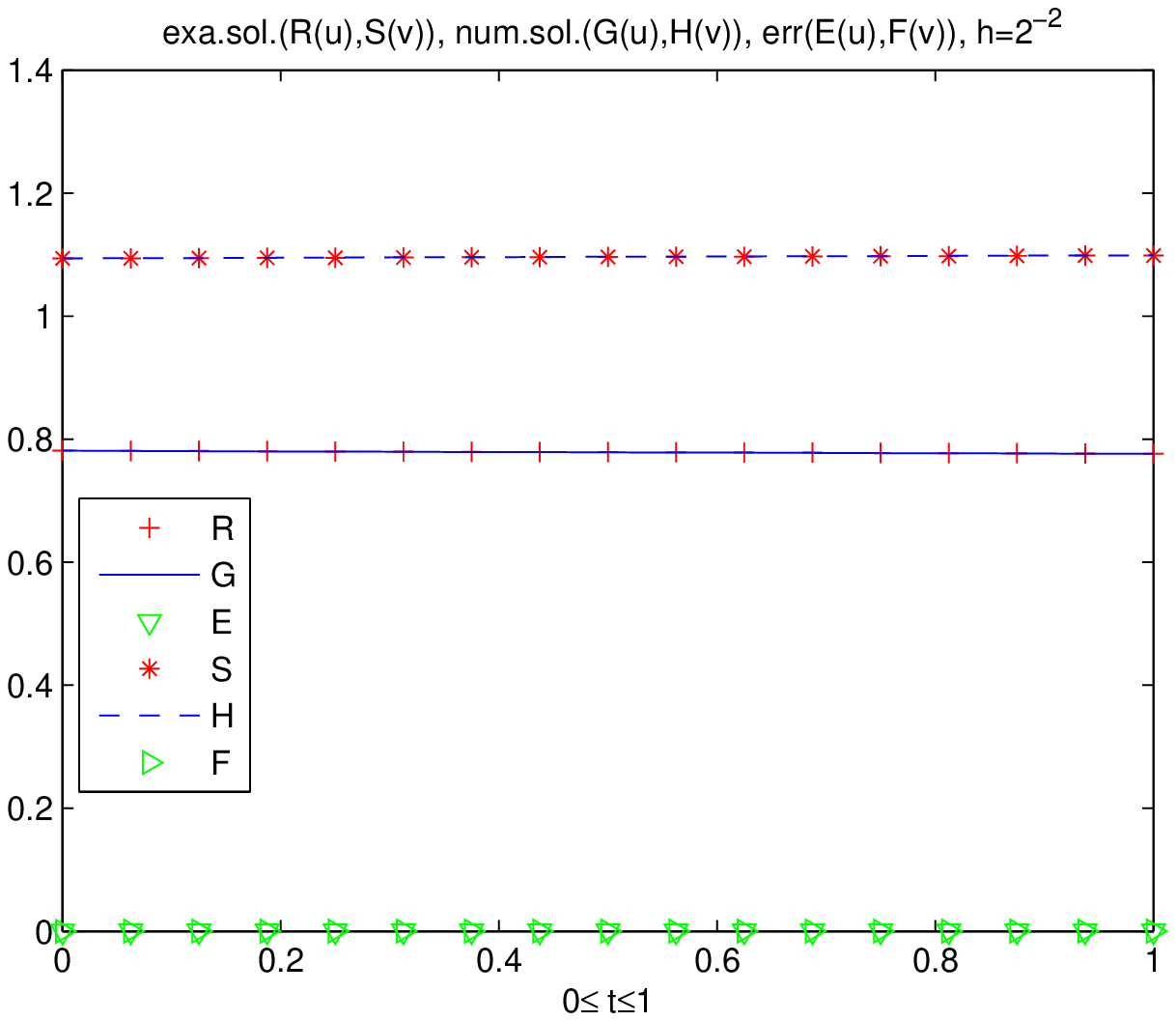,width=7cm}\\
               \psfig{file=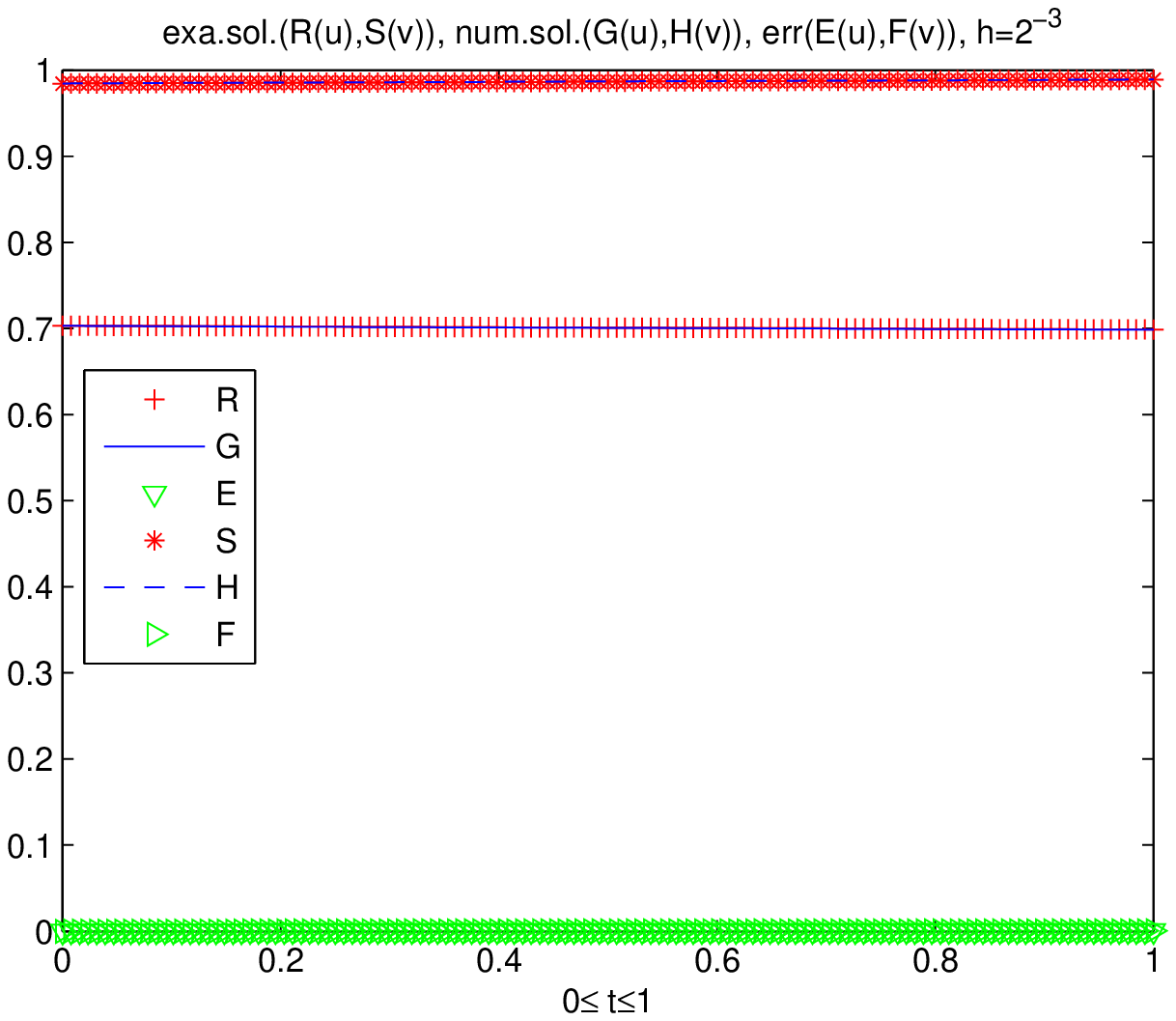,width=7cm} & \psfig{file=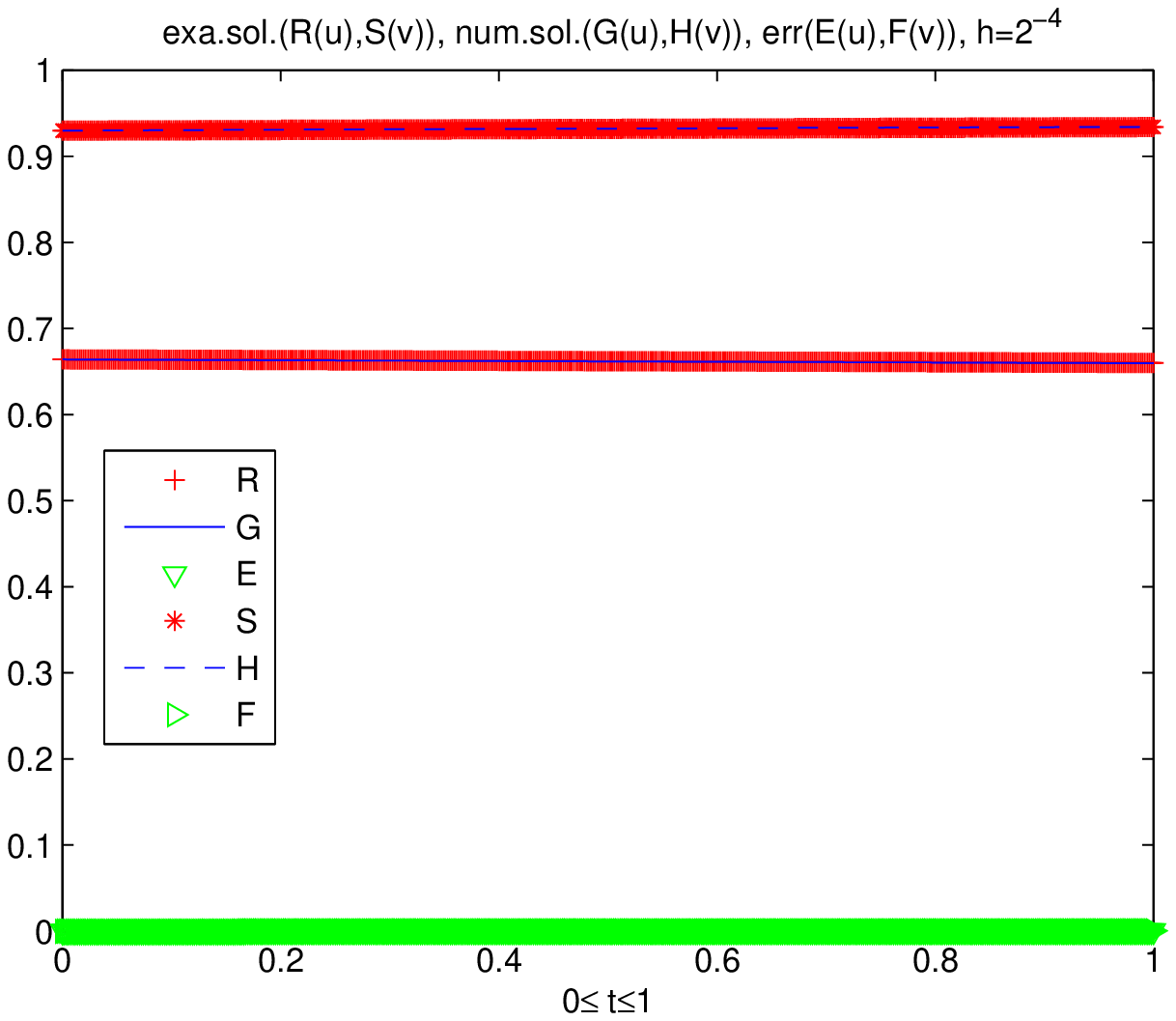,width=7cm}\\
            \end{tabular}
        \end{center}
         \caption{Graphs corresponding to a three-level time-split MacCormack method}
          \label{fig1}
          \end{figure}

          \begin{figure}
         \begin{center}
          Stability and convergence of a three-level time-split MacCormack method: $a=\mu=1$.
          \begin{tabular}{cc}
              \psfig{file=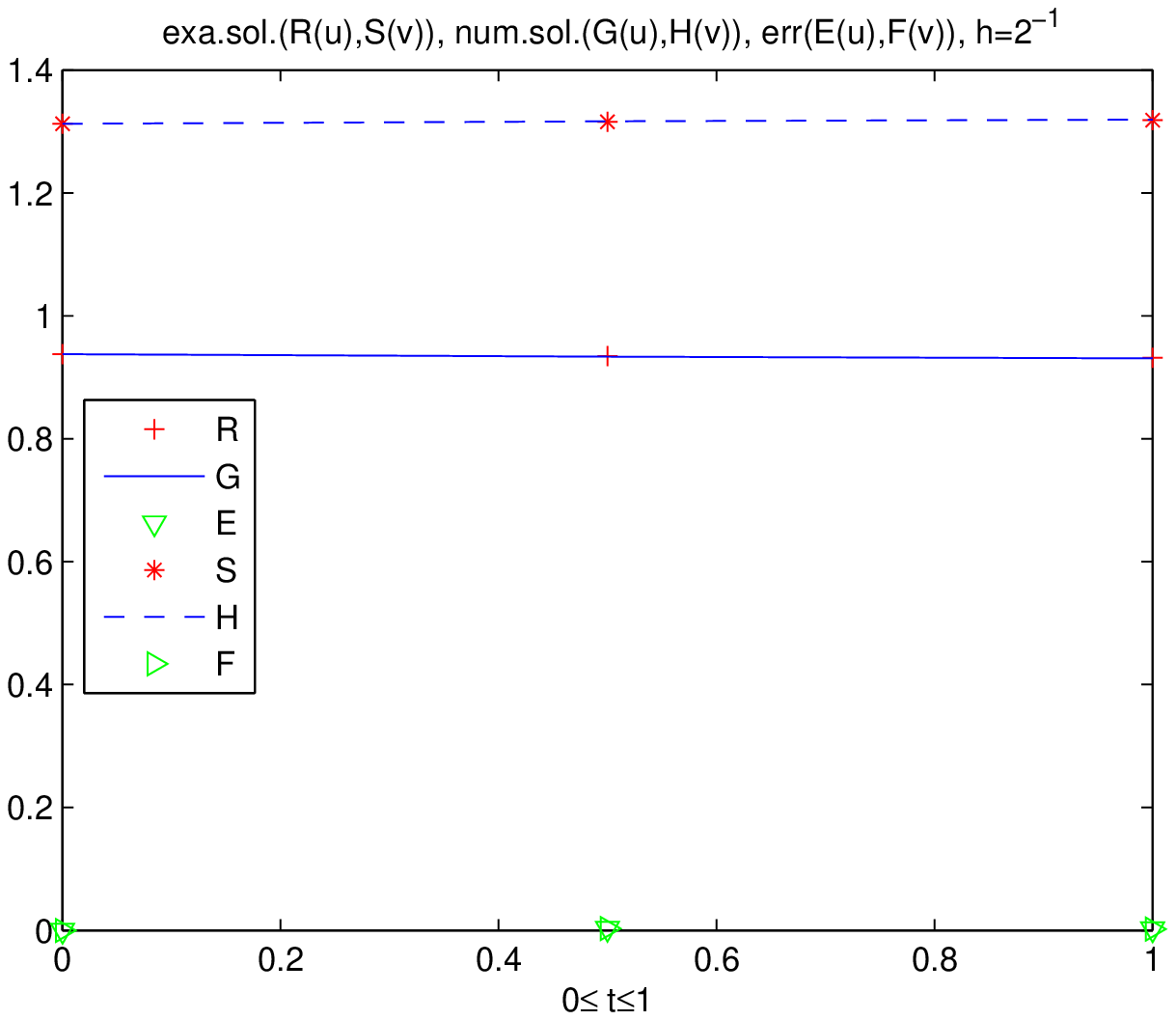,width=7cm} & \psfig{file=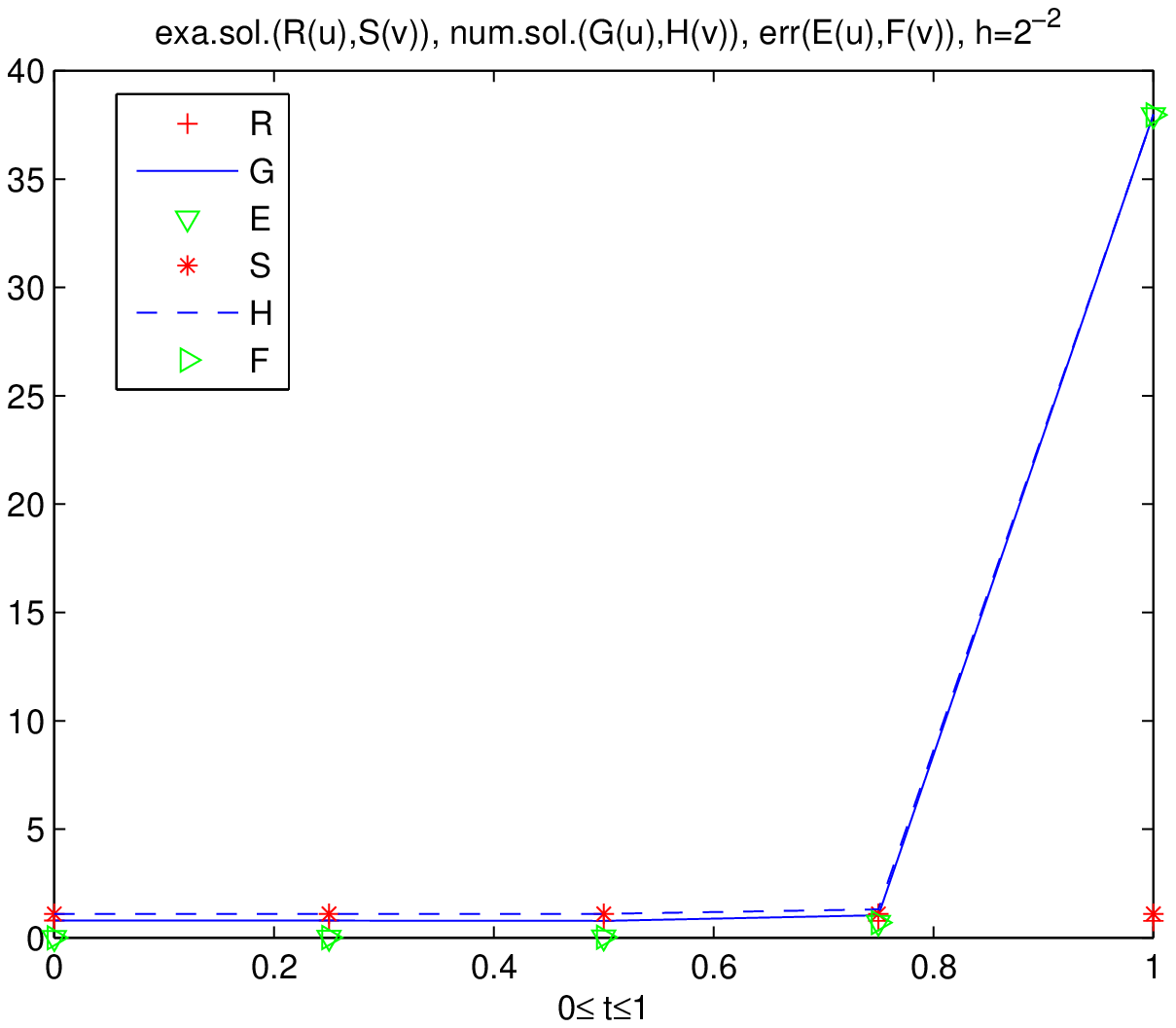,width=7cm}\\
            \end{tabular}
        \end{center}
         \caption{Graphs corresponding to a three-level time-split MacCormack method}
          \label{fig2}
          \end{figure}
          \begin{figure}
         \begin{center}
          Stability and convergence of a three-level time-split MacCormack method: $a=\mu=1$.
               \psfig{file=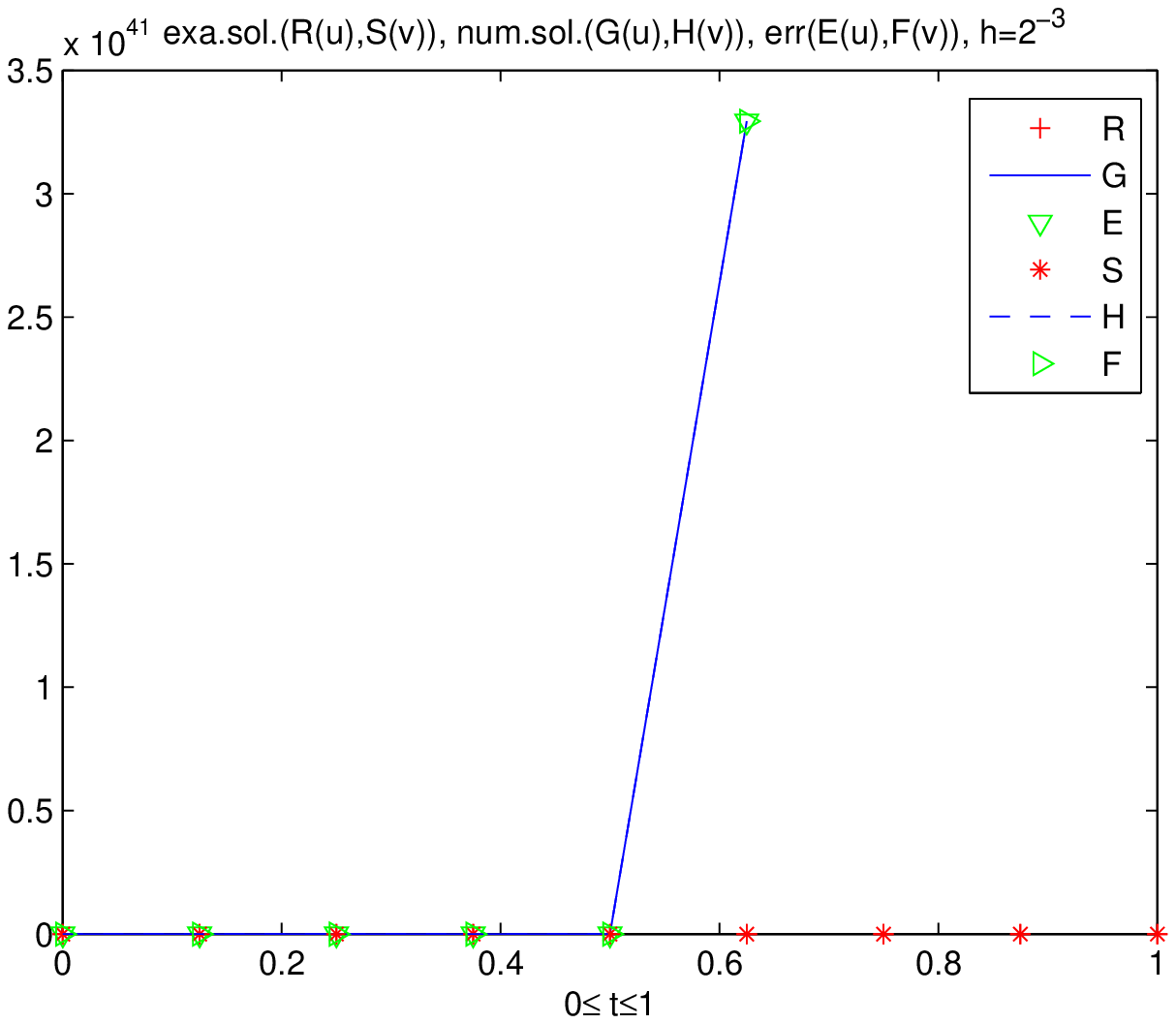,width=7cm}
        \end{center}
         \caption{Graphs corresponding to a three-level time-split MacCormack method}
          \label{fig3}
          \end{figure}

          \begin{figure}
         \begin{center}
          Stability and convergence of a three-level time-split MacCormack method: $a=\mu=1$.
          \begin{tabular}{cc}
              \psfig{file=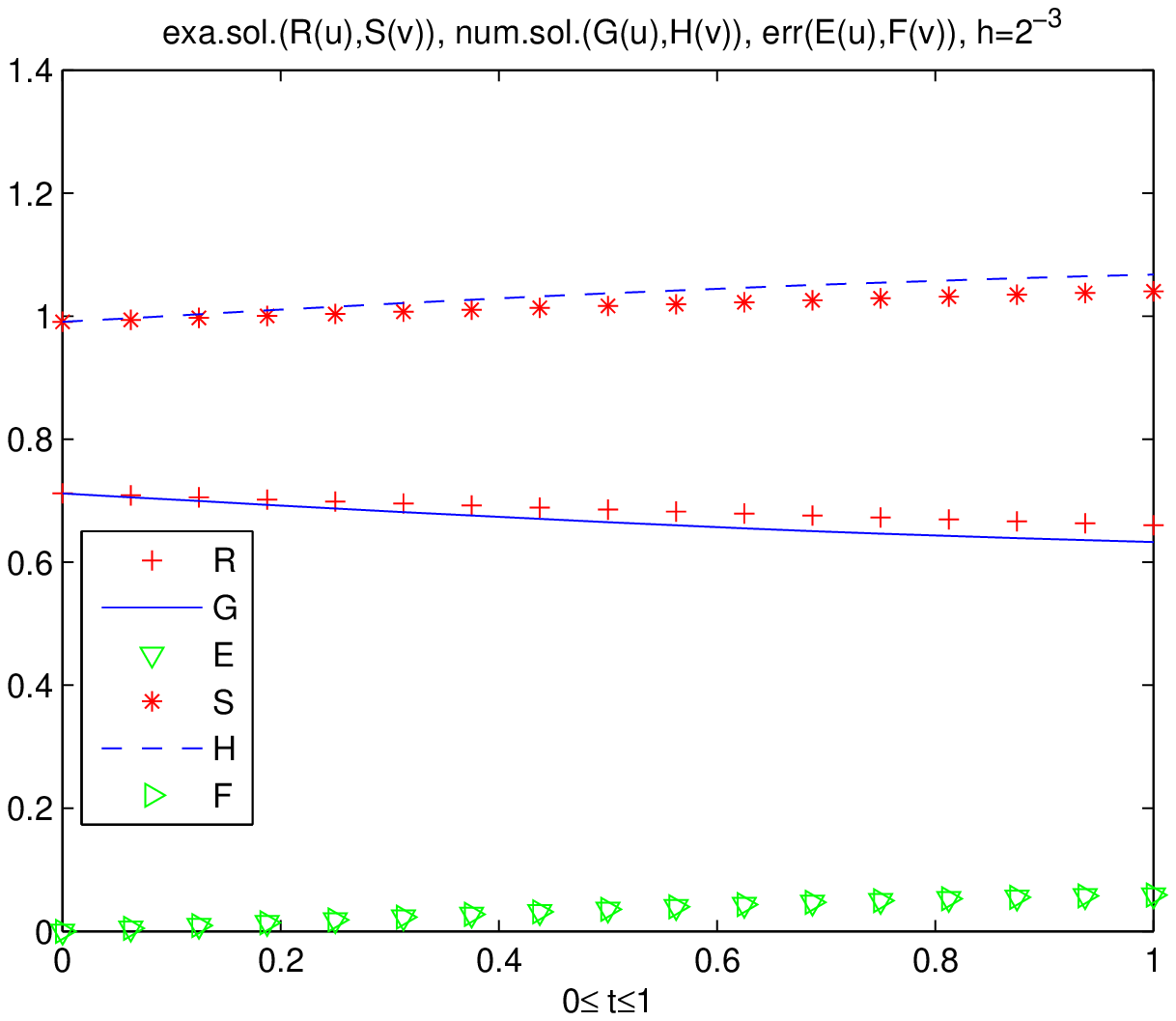,width=7cm} & \psfig{file=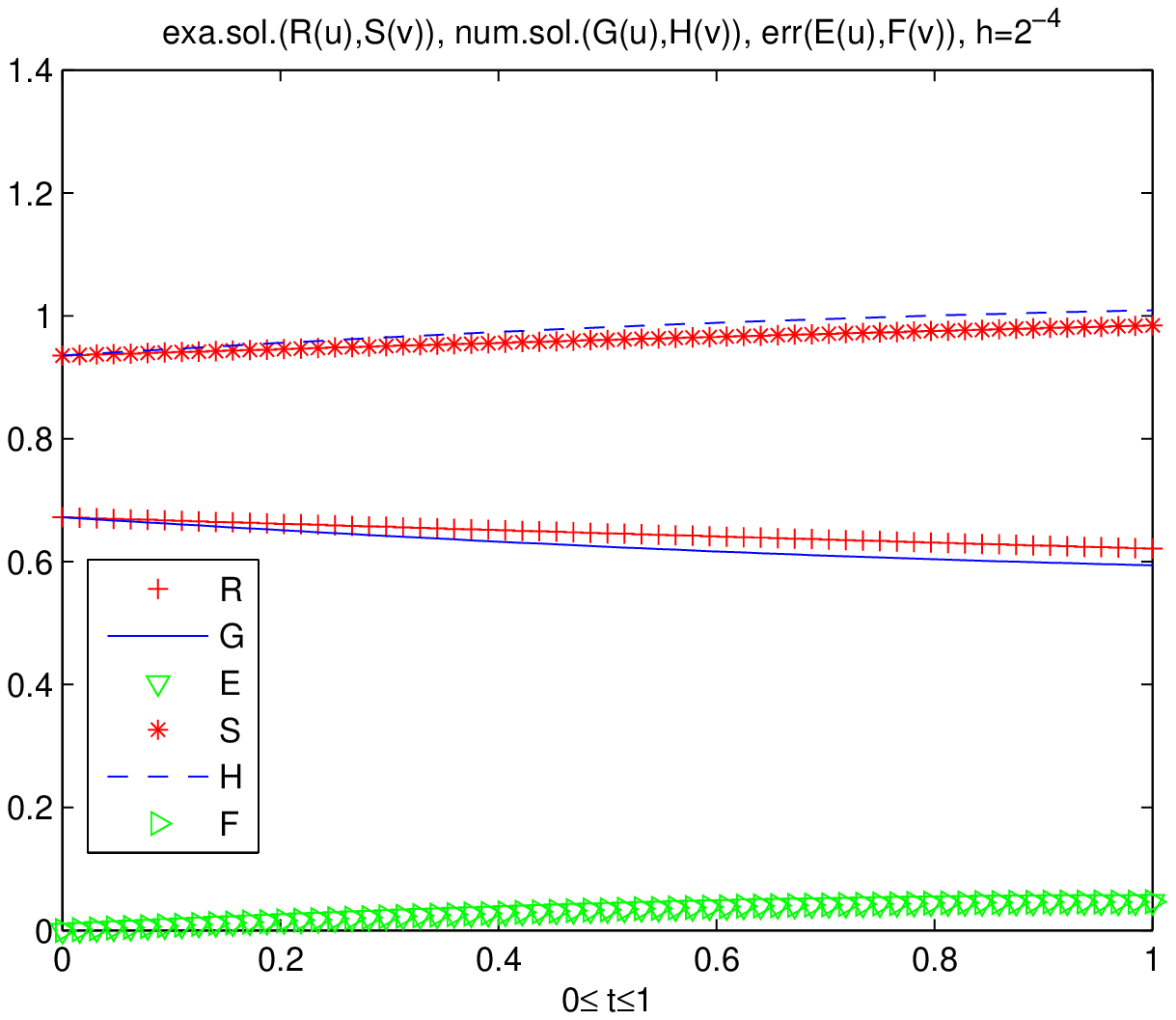,width=7cm}\\
              \psfig{file=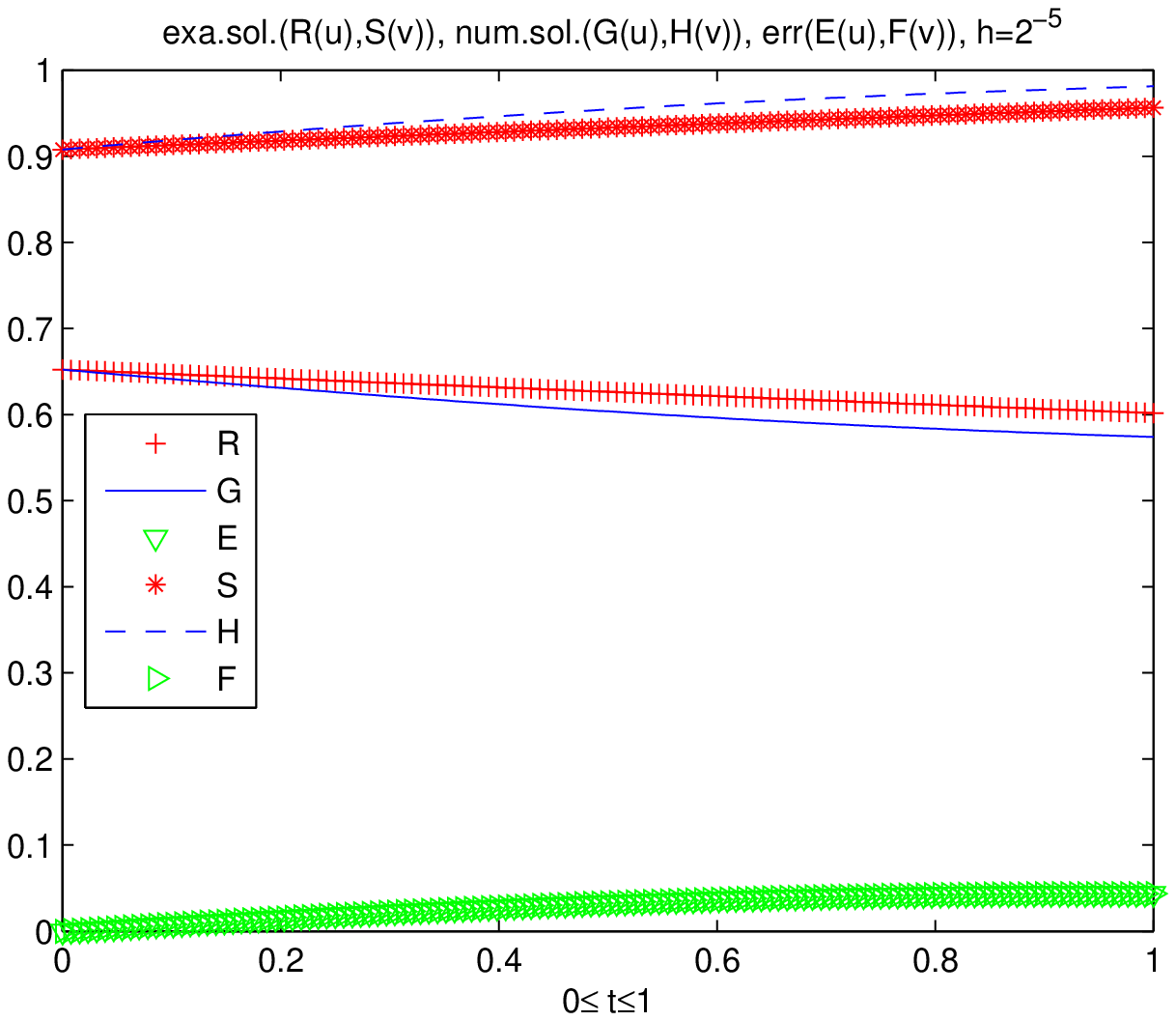,width=7cm} & \psfig{file=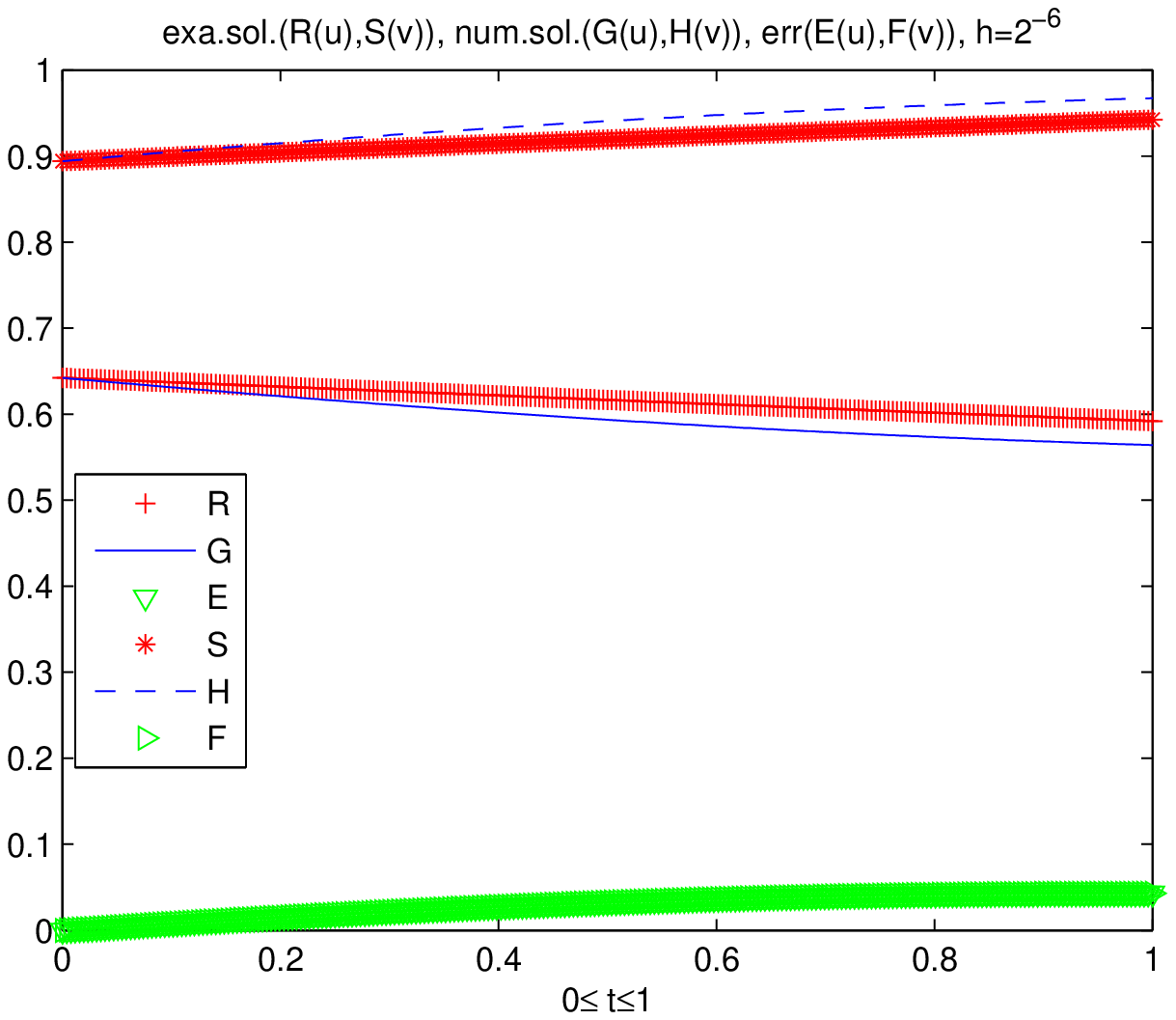,width=7cm}\\
            \end{tabular}
        \end{center}
         \caption{Graphs corresponding to a three-level time-split MacCormack method}
          \label{fig4}
          \end{figure}

     \end{document}